\documentclass[bj,preprint]{imsart}
\RequirePackage[OT1]{fontenc}
\RequirePackage{amsthm,amsmath}
\RequirePackage[numbers]{natbib}
\RequirePackage[colorlinks,citecolor=blue,urlcolor=blue]{hyperref}
\usepackage{amsmath,amsthm,amssymb,amsfonts,bm,bbm,nicefrac,array,color,graphicx,fancyhdr,lastpage,mathtools,titlesec} 
\usepackage[margin=3cm,top=4cm]{geometry}
\usepackage[shortlabels]{enumitem}
\usepackage{cmap} \usepackage{lmodern} \usepackage[T1]{fontenc} 
\definecolor{navy}{RGB}{20,0,105}
\titleformat*{\subsection}{\large\bfseries}
\titleformat*{\subsubsection}{\bfseries}
\usepackage{cleveref}		
\usepackage[titletoc,title]{appendix}

\startlocaldefs

\newcommand{\dif}{\mathop{}\!\mathrm{d}} 
\newcommand{\dt}{\dif t} \newcommand{\ds}{\dif s}  \newcommand{\dX}{\dif X} \newcommand{\dY}{\dif Y}  \newcommand{\dW}{\dif W} \newcommand{\dx}{\dif x}  \newcommand{\dy}{\dif y} \newcommand{\dz}{\dif z}  \newcommand{\dPi}{\dif \Pi}    \newcommand{\dmu}{\dif \mu} 
 
\newcommand{\od}[2]{\frac{\dif #1}{\dif #2}}
\newcommand{\PP}{\mathbb{P}}  \newcommand{\QQ}{\mathbb{Q}} \newcommand{\NN}{\mathbb{N}} \newcommand{\ZZ}{\mathbb{Z}} \newcommand{\RR}{\mathbb{R}}   \newcommand{\II}{\mathbbm{1}}
\newcommand{\Aa}{\mathcal{A}} \newcommand{\Bb}{\mathcal{B}} \newcommand{\Cc}{\mathcal{C}} \newcommand{\Dd}{\mathcal{D}}  \newcommand{\Ff}{\mathcal{F}} \newcommand{\Gg}{\mathcal{G}}  \newcommand{\Ii}{\mathcal{I}}          \newcommand{\Ss}{\mathcal{S}}     \newcommand{\Xx}{\mathcal{X}}  
\DeclarePairedDelimiter{\norm}{\lVert}{\rVert}
\DeclarePairedDelimiter{\abs}{\lvert}{\rvert}
\DeclarePairedDelimiter{\braces}{ \{ }{ \} }
\DeclarePairedDelimiter{\brackets}{(}{)}
\DeclarePairedDelimiter{\sqbrackets}{[}{]}
\DeclarePairedDelimiter{\qv}{\langle}{\rangle}
\DeclarePairedDelimiter{\ip}{\langle}{\rangle}

\DeclareMathOperator*{\argmin}{argmin}

\DeclareMathOperator{\Var}{Var}
\DeclareMathOperator{\Cov}{Cov}

\DeclareMathOperator{\Span}{span}
\DeclareMathOperator{\KL}{KL} 
\newcommand{\eps}{\varepsilon}
\newcommand{\st}{:} 
\newcommand{\per}{\text{per}}
\newcommand{\iidsim}{\overset{iid}{\sim}}
\newcommand{\WW}{\mathbb{W}}
\newcommand{\BB}{\mathbb{B}}
\theoremstyle{definition}   \newtheorem{assumption}{Assumption}
\theoremstyle{remark} \newtheorem{example}{Example} \newtheorem*{remark}{Remark}  \newtheorem*{remarks}{Remarks}
\theoremstyle{plain} \newtheorem{theorem}{Theorem} \newtheorem*{theorem*}{Theorem} \newtheorem{lemma}[theorem]{Lemma}  \newtheorem*{lemma*}{Lemma} \newtheorem{proposition}[theorem]{Proposition}
\crefname{appsec}{Appendix}{Appendices}
\crefname{assumption}{assumption}{assumptions}
\crefname{equation}{}{}
\crefname{enumi}{}{}
\newlist{lemenum}{enumerate}{1} 
\setlist[lemenum]{label=\roman*., ref=\arabic{theorem}(\roman*)}
\crefalias{lemenumi}{lemma} 
\newlist{thmenum}{enumerate}{1} 
\setlist[thmenum]{label=\Alph*., ref=\arabic{theorem}\Alph*}
\crefalias{thmenumi}{theorem} 

\endlocaldefs

\begin{document}

\begin{frontmatter}
\title{Nonparametric Bayesian posterior contraction rates for scalar diffusions with high-frequency data}
\runtitle{Posterior contraction for high-frequency sampled diffusions}

\begin{aug}
\author{\fnms{Kweku} \snm{Abraham}\ead[label=e1]{lkwabraham@statslab.cam.ac.uk}}

\address{Statistical Laboratory, Department of Pure Mathematics and Mathematical Statistics, University of Cambridge, Wilberforce Road, Cambridge CB3 0WB, UK.
\printead{e1}}

\runauthor{K. Abraham}

\affiliation{University of Cambridge}

\end{aug}

\begin{abstract}
We consider inference in the scalar diffusion model $\dX_t=b(X_t)\dt+\sigma(X_t)\dW_t$ with discrete data $(X_{j\Delta_n})_{0\leq j \leq n}$, $n\to \infty,~\Delta_n\to 0$ and periodic coefficients. For $\sigma$ given, we prove a general theorem detailing conditions under which Bayesian posteriors will contract in $L^2$--distance around the true drift function $b_0$ at the frequentist minimax rate (up to logarithmic factors) over Besov smoothness classes. We exhibit natural nonparametric priors which satisfy our conditions. Our results show that the Bayesian method adapts both to an unknown sampling regime and to unknown smoothness.
\end{abstract}


\end{frontmatter}

\section{Introduction}\label{sec:Introduction}
Consider a scalar diffusion process $(X_t)_{t\geq 0}$ starting at some $X_0$ and evolving according to the stochastic differential equation 
\begin{equation*}  \dX_t=b(X_t)\dt+\sigma(X_t)\dW_t,\end{equation*}
where $W_t$ is a standard Brownian motion.
It is of considerable interest to estimate the parameters $b$ and $\sigma$, which are arbitrary functions (until we place further assumptions on their form), so that the model is naturally \emph{nonparametric}. As we will explain in \Cref{sec:FrameworkAndAssumptions}, the problems of estimating $\sigma$ and $b$ can essentially be decoupled in the setting to be considered here, so in this paper we consider estimation of the drift function $b$ when the diffusion coefficient $\sigma$ is assumed to be given.

It is realistic to assume that we do not observe the full trajectory $(X_t)_{t\leq T}$ but rather the process sampled at discrete time intervals $(X_{k\Delta})_{k\leq n}$.
The estimation problem for $b$ and $\sigma$ has been studied extensively and minimax rates have been attained in two sampling frameworks: \emph{low-frequency}, where $\Delta$ is fixed and asymptotics are taken as $n\to\infty$ (see Gobet--Hoffmann--Reiss \cite{Gobet2004}), and \emph{high-frequency}, where asymptotics are taken as $n\to \infty$ and $\Delta=\Delta_n\to 0$, typically assuming also that $n\Delta^2\to 0$ and $n\Delta\to \infty$ (see Hoffmann \cite{Hoffmann1999b}, Comte et al.\ \cite{Comte2007}). See also eg.\ \cite{Dalalyan2005}, \cite{Gugushvili2014}, \cite{Pokern2013}, \cite{vanderMeulen2017} for more papers addressing nonparametric estimation for diffusions. 

For typical frequentist methods, one must know which sampling regime the data is drawn from. In particular, the low-frequency estimator from \cite{Gobet2004} is consistent in the high-frequency setting but numerical simulations suggest it does not attain the minimax rate (see the discussion in Chorowski \cite{Chorowski2016}), while the high-frequency estimators of \cite{Hoffmann1999b} and \cite{Comte2007} are not even consistent with low-frequency data. The only previous result known to the author regarding adaptation to the sampling regime in the nonparametric setting is found in \cite{Chorowski2016}, where Chorowski is able to estimate the diffusion coefficient $\sigma$ but not the drift, and obtains the minimax rate when $\sigma$ has 1 derivative but not for smoother diffusion coefficients. 

For this paper we consider estimation of the parameters in a diffusion model from a nonparametric Bayesian perspective. Bayesian methods for diffusion estimation can be implemented in practice (eg.\ see Papaspiliopoulos et al.\ \cite{Papaspiliopoulos2012}). For Bayesian estimation, the statistician need only specify a prior, and for estimating diffusions from discrete samples the  prior need not reference the sampling regime, so Bayesian methodology provides a natural candidate for a unified approach to the high- and low-frequency settings. Our results imply that Bayesian methods can adapt both to the sampling regime and also to unknown smoothness of the drift function (see the remarks after \Cref{prop:InvariantDensityPrior} and \Cref{prop:BasicSievePrior} respectively for details). These results are proved under the frequentist assumption of a fixed true parameter, so this paper belongs to the field of \emph{frequentist analysis of Bayesian procedures}. See, for example, Ghosal \& van der Vaart \cite{Ghosal2017} for an introduction to this field. 

It has previously been shown that in the low-frequency setting we have a \emph{posterior contraction rate}, guaranteeing that posteriors corresponding to reasonable priors concentrate their mass on neighbourhoods of the true parameter shrinking at the fastest possible rate (up to log factors) -- see Nickl \& S{\"o}hl \cite{Nickl2017}. To complete a proof that such posteriors contract at a rate adapting to the sampling regime, it remains to prove a corresponding contraction rate in the high-frequency setting. This forms the key contribution of the current paper: we prove that a large class of ``reasonable'' priors will exhibit posterior contraction at the optimal rate (up to log factors) in $L^2$--distance. This in turn guarantees that point estimators based on the posterior will achieve the frequentist minimax optimal rate (see the remark after \Cref{thm:ConcretePriorContraction}) in both high- and low-frequency regimes.

The broad structure of the proof is inspired by that in \cite{Nickl2017}: we use the testing approach of Ghosal--Ghosh--van der Vaart \cite{GGV2000}, coupled with the insight of Gin{\'e} and Nickl \cite{Gine2011a} that one may prove the existence of the required tests by finding an estimator with good enough concentration around the true parameter. The main ingredients here are:
\begin{itemize}	
\item A concentration inequality for a (frequentist) estimator, from which we construct tests of the true $b_0$ against a set of suitable (sufficiently separated) alternatives. 
See \Cref{sec:Concentration}.
\item A small ball result, to relate the $L^2$--distance to the information-theoretic Kullback--Leibler ``distance''. See \Cref{sec:SmallBallProbs}.
\end{itemize}
Though the structure reflects that of \cite{Nickl2017} the details are very different. Estimators for the low-frequency setting are typically based on the mixing properties of $(X_{k\Delta})$ viewed as a Markov chain and the spectral structure of its transition matrix (see Gobet--Hoffmann--Reiss \cite{Gobet2004}) and fail to take full advantage of the local information one sees when $\Delta\to 0$. Here we instead use an estimator introduced in Comte et al.\ \cite{Comte2007} which uses the assumption $\Delta\to 0$ to view estimation of $b$ as a regression problem. To prove this estimator concentrates depends on a key insight of this paper: the Markov chain concentration results used in the low-frequency setting (which give \emph{worse} bounds as $\Delta\to 0$) must be supplemented by H{\"o}lder type continuity results, which crucially rely on the assumption $\Delta\to0$. We further supplement by martingale concentration results.

Similarly, the small ball result in the low-frequency setting depends on Markov chain mixing. Here, we instead adapt the approach of van der Meulen \& van Zanten \cite{vanderMeulenvanZanten2013}. They demonstrate that the Kullback--Leibler divergence in the discrete setting can be controlled by the corresponding divergence in the continuous data model; a key new result of the current paper is that in the high-frequency setting this control extends to give a bound on the variance of the log likelihood ratio. 

As described above, a key attraction of the Bayesian method is that it allows the statistician to approach the low- and high-frequency regimes in a unified way. Another attraction is that it naturally suggests uncertainty quantification via posterior credible sets. The contraction rate theorems proved in this paper and \cite{Nickl2017} are not by themselves enough to prove that credible sets behave as advertised. For that one may aim for a nonparametric Bernstein--von Mises result -- see for example Castillo \& Nickl \cite{Castillo2013,Castillo2014}. The posterior contraction rate proved here constitutes a key first step towards a proof of a Bernstein--von Mises result for the high-frequency sampled diffusion model, since it allows one to localise the posterior around the true parameter, as in the proofs in Nickl \cite{Nickl2017b} for a non-linear inverse problem comparable to the problem here.

\section{Framework and assumptions}\label{sec:FrameworkAndAssumptions}
The notation introduced throughout the paper is gathered in \Cref{sec:notation}.

We work with a scalar diffusion process $(X_t)_{t\geq 0}$ starting at some $X_0$ and evolving according to the stochastic differential equation 
\begin{equation} \label{eqn:sde} \dX_t=b(X_t)\dt +\sigma(X_t)\dW_t, \end{equation} for $W_t$ a standard Brownian motion. The parameters $b$ and $\sigma$ are assumed to be 1--periodic and we also assume the following.

\begin{assumption}\label{assumptions:sigma}
$\sigma\in C_\per^2([0,1])$ is given. 
Continuity guarantees the existence of an upper bound $\sigma_U< \infty$ and we further assume the existence of a lower bound $\sigma_L>0$ so that $\sigma_L\leq \sigma(x)\leq \sigma_U$ for all $x\in [0,1]$. Here $C_\per^2([0,1])$ denotes $C^2([0,1])$ functions with periodic boundary conditions (i.e.\ $\sigma(0)=\sigma(1)$, $\sigma'(0)=\sigma'(1)$ and $\sigma''(0)=\sigma''(1)$). 
\end{assumption}

\begin{assumption}\label{assumptions:b}
$b$ is continuously differentiable with given norm bound. Precisely, we assume $b\in\Theta$, where, for some arbitrary but known constant $K_0,$ \[\Theta=\Theta(K_0)=\braces{f\in C_\per^1([0,1]) \st ~\norm{f}_{C_\per^1}=\norm{f}_\infty+\norm{f'}_\infty\leq K_0}.\] ($\norm{\cdot}_\infty$ denotes the supremum norm, $\norm{f}_\infty=\sup_{x\in[0,1]}\abs{f(x)}$.) Note in particular that $K_0$ upper bounds $\norm{b}_\infty$ and that $b$ is Lipschitz continuous with constant at most $K_0$. 

$\Theta$ is the maximal set over which we prove contraction, and we will in general make the stronger assumption that in fact $b\in \Theta_s(A_0)$, where 
\[\Theta_s(A_0):=\braces{f\in \Theta : \norm{f}_{B_{2,\infty}^s}\leq A_0<\infty},\quad A_0>0,~s\geq 1\]
 with $B_{p,q}^s$ denoting a periodic Besov space and $\norm{\cdot}_{B_{p,q}^s}$ denoting the associated norm: see \Cref{sec:ApproxSpaces} for a definition of the periodic Besov spaces we use (readers unfamiliar with Besov spaces may substitute the $L^2$--Sobolev space $H^s=B_{2,2}^s\subseteq B_{2,\infty}^s$ for $B_{2,\infty}^s$ and only mildly weaken the results). We generally assume the regularity index $s$ is unknown. Our results will therefore aim to be \emph{adaptive}, at least in the smoothness index (to be fully adaptive we would need to adapt to $K_0$ also).  
\end{assumption}

Under \Cref{assumptions:b,assumptions:sigma}, there is a unique strong solution to \Cref{eqn:sde} (see, for example, Bass \cite{Bass2011} Theorem 24.3). Moreover, this solution is also weakly unique (= unique in law) and satisfies the Markov property (see \cite{Bass2011} Proposition 25.2 and Theorem 39.2). We denote by $P_b^{(x)}$ the law (on the cylindrical $\sigma$--algebra of $C([0,\infty])$) of the unique solution of \Cref{eqn:sde} started from $X_0=x$.

We consider ``high-frequency data'' $(X_{k\Delta_n})_{k=0}^n$ sampled from this solution, where asymptotics are taken as $n\to \infty$, with $\Delta_n\to 0$ and $n\Delta_n \to \infty$. 
We will suppress the subscript and simply write $\Delta$ for $\Delta_n$.  
Throughout we will write $X^{(n)}=(X_0,\dots,X_{n\Delta})$ as shorthand for our data and similarly we write $x^{(n)}=(x_0,\dots,x_{n\Delta})$. We will denote by $\Ii$ the set $\braces{K_0,\sigma_L,\sigma_U}$ so that, for example, $C(\Ii)$ will be a constant depending on these parameters.

Beyond guaranteeing existence and uniqueness of a solution, our assumptions also guarantee the existence of transition densities for the discretely sampled process (see Gihman \& Skorohod \cite{Gihman1972} Theorem 13.2 for an explicit formula for the transition densities). Morever, there also exists an invariant distribution $\mu_b$, with density $\pi_b$, for the periodised process $\dot{X}=X\mod 1$. Defining $I_b(x)=\int_0^x \frac{2b}{\sigma^2}(y)\dy$ for $x\in[0,1],$ the density is   
\begin{align*} &\pi_b(x) =\frac{e^{I_b(x)}}{H_b\sigma^2(x)}\brackets[\Big]{e^{I_b(1)}\int_x^1 e^{-I_b(y)}\dy +\int_0^x e^{-I_b(y)}\dy}, \qquad x\in [0,1], \\ & H_b=\int_0^1 \frac{e^{I_b(x)}}{\sigma^2(x)}\brackets[\Big]{e^{I_b(1)}\int_x^1 e^{-I_b(y)}\dy +\int_0^x e^{-I_b(y)}\dy}\dx, \end{align*}
(see Bhattacharya et al.\ \cite{Bhattacharya1999}, equations 2.15 to 2.17; note we have chosen a different normalisation constant so the expressions appear slightly different).

Observe that $\pi_b$ is bounded uniformly away from zero and infinity, i.e.\ there exist constants $0<\pi_L,\pi_U<\infty$ depending only on $\Ii$ so that for any $b\in \Theta$ and any $x\in [0,1]$ we have $\pi_L\leq\pi_b(x)\leq \pi_U$. 
Precisely, we see that $\sigma_U^{-2}e^{-6K_0\sigma_L^{-2}}\leq H_b\leq \sigma_L^{-2}e^{6K_0\sigma_L^{-2}},$
and we deduce we can take $\pi_L=\pi_U^{-1}=\sigma_L^2\sigma_U^{-2} e^{-12K_0\sigma_L^{-2}}.$

We assume that $X_0\in [0,1)$ and that $X_0=\dot{X}_0$ follows this invariant distribution. 
\begin{assumption}\label{assumption:invariant}
$X_0\sim \mu_b$. 
\end{assumption}
We will write $P_b$ for the law of the full process $X$ under \Cref{assumptions:b,assumptions:sigma,assumption:invariant}, and we will write $E_b$ for expectation according to this law. Note $\mu_b$ is not invariant for $P_b$, but nevertheless $E_b(f(X_t))=E_b(f(X_0))$ for any 1--periodic function $f$ (eg.\ see the proof of \Cref{thm:concentration-for-diffusions}). Since we will be estimating the 1--periodic function $b$, the assumption that $X_0\in [0,1)$ is unimportant.

Finally, we need to assume that $\Delta\to0$ at a fast enough rate.
\begin{assumption}\label{assumption:Delta}
$n\Delta^2 \log(1/\Delta)\leq L_0$ for some (unknown) constant $L_0$. Since we already assume $n\Delta\to\infty$, this new assumption is equivalent to $n\Delta^2\log(n)\leq L_0'$ for some constant $L_0'$. 
\end{assumption}

Throughout we make the frequentist assumption that the data is generated according to some fixed true parameter denoted $b_0$. We use $\mu_0$ as shorthand for $\mu_{b_0}$, and similarly for $\pi_0$ and so on. Where context allows, we write $\mu$ for $\mu_b$ with a generic drift $b$. 

\begin{remarks}[Comments on assumptions]
\emph{Periodicity assumption.} We assume $b$ and $\sigma$ are periodic so that we need only estimate $b$ on $[0,1]$. One could alternatively assume $b$ satisfies some growth condition ensuring recurrence, then estimate the restriction of $b$ to $[0,1]$, as in Comte et al.\ \cite{Comte2007} and van der Meulen \& van Zanten \cite{vanderMeulenvanZanten2013}. The proofs in this paper work in this alternative framework with minor technical changes, provided one assumes the behaviour of $b$ outside $[0,1]$ can be exactly matched by a draw from the prior. 

\emph{Assuming that $\sigma\in C^2_\per$ is given.} 
If we observe continuous data $(X_t)_{t\leq T}$ then $\sigma$ is known exactly (at least at any point visited by the process) via the expression for the quadratic variation $\qv{X}_t=\int_0^t \sigma^2(X_s)\ds$. With high-frequency data we cannot perfectly reconstruct the diffusion coefficient from the data, but we can estimate it at a much faster rate than the drift. When $b$ and $\sigma$ are both assumed unknown, if $b$ is $s$-smooth and $\sigma$ is $s'$-smooth, the minimax errors for $b$ and $\sigma$ respectively scale as $(n\Delta)^{-s/(1+2s)}$ and $n^{-s'/(1+2s')}$, as can be shown by slightly adapting Theorems 5 and 6 from Hoffmann \cite{Hoffmann1999b} so that they apply in the periodic setting we use here. Since we assume that $n\Delta^2\to 0$, it follows that $n\Delta\leq n^{1/2}$ for large $n$, hence we can estimate $\sigma$ at a faster rate than $b$ regardless of their relative smoothnesses. 

Further, note that the problems of estimating $b$ and $\sigma$ in the high-frequency setting are essentially independent. For example, the smoothness of $\sigma$ does not affect the rate for estimating $b$, and vice-versa -- see \cite{Hoffmann1999b}. We are therefore not substantially simplifying the problem of estimating $b$ through the assumption that $\sigma$ is given.

The assumption that $\sigma^2$ is twice differentiable is a typical minimal assumption to ensure transition densities exist.

\emph{Assuming a known bound on $\norm{b}_{C_\per^1}$.}
The assumption that $b$ has one derivative is a typical minimal assumption to ensure that the diffusion equation \Cref{eqn:sde} has a strong solution and that this solution has an invariant density. The assumption of a \emph{known} bound for the $C_\per^1$--norm of the function is undesirable, but needed for the proofs, in particular to ensure the existence of a uniform lower bound $\pi_L$ on the invariant densities. This lower bound is essential for the Markov chain mixing results as its reciprocal controls the mixing time in \Cref{thm:concentration-for-diffusions}. It is plausible that needing this assumption is inherent to the problem rather than an artefact of the proofs: possible methods to bypass the Markov chain mixing arguments, such as the martingale approach of \cite{Comte2007} Lemma 1, also rely on such a uniform lower bound. One could nonetheless hope that our results apply to an unbounded prior placing sufficient weight on $\Theta(K_n)$ for some slowly growing sequence $K_n$, but the lower bound $\pi_L$ scales unfavourably as $e^{-K_n}$, which rules out this approach.

These boundedness assumptions in principle exclude Gaussian priors, which are computationally attractive. In practice, one could choose a very large value for $K_0$ and approximate Gaussian priors arbitrarily well using truncated Gaussian priors.

\emph{Assuming $X_0\sim \mu_b$.}
It can be shown (see the proof of \Cref{thm:concentration-for-diffusions}) that the law of $\dot{X}_t$ converges to $\mu_b$ at exponential rate from any starting distribution, so assuming $X_0\sim \mu_b$ is not restrictive (as mentioned, our fixing $X_0\in [0,1)$ is arbitrary but unimportant).


\emph{Assuming $n\Delta^2\log(1/\Delta)\leq L_0$.}
It is typical in the high-frequency setting to assume $n\Delta^2\to 0$ (indeed the minimax rates in \cite{Hoffmann1999b} are only proved under this assumption) but for technical reasons in the concentration section (\Cref{sec:ComteConcentration}) 
we need the above. 
\end{remarks}

\subsection{Spaces of approximation} \label{sec:ApproxSpaces}
We will throughout depend on a family $\braces{S_m : m\in \NN\cup \braces{0}}$ of function spaces. For our purposes we will take the $S_m$ to be periodised Meyer-type wavelet spaces \[S_m=\Span\brackets{\braces{\psi_{lk} : 0\leq k< 2^l,0\leq l < m}\cup\braces{1}}.\] We will denote $\psi_{-1,0}\equiv 1$ for convenience. Denote by $\ip{\cdot,\cdot}$ the $L^2([0,1])$ inner product and by $\norm{\cdot}_2$ the $L^2$--norm,  i.e.\
$\ip{f,g}=\int_0^1 f(x)g(x)\dx$ and $\norm{f}_2=\ip{f,f}^{1/2}$ for $f,g\in L^2([0,1]). $ 
One definition of the (periodic) Besov norm $\norm{f}_{B_{2,\infty}^s}$ is, for $f_{lk}:=\ip{f,\psi_{lk}}$, 
\begin{equation}\label{eqn:B2inftywaveletcharacterisation}\norm{f}_{B_{2,\infty}^s}= \abs{f_{-1,0}}+\sup_{l\geq 0} 2^{ls}\brackets[\bigg]{\sum_{k=0}^{2^l-1} f_{lk}^2}^{1/2},\end{equation} with $B_{2,\infty}^s$ defined as those periodic $f\in L^2([0,1])$ for which this norm is finite. See  Gin{\'e} \& Nickl \cite{Gine2016} Sections 4.2.3 and 4.3.4 for a construction of periodised Meyer-type wavelets and a proof that this wavelet norm characterisation agrees with other possible definitions of the desired Besov space. 

Note that the orthonormality of the wavelet basis means $\norm{f}_2^2=\sum_{l,k} f_{lk}^2$. Thus it follows from the above definition of the Besov norm that for any $b\in B_{2,\infty}^s([0,1])$ we have \begin{equation} \label{eqn:ApproxSpaces}\norm{\pi_m b-b}_2 \leq K\norm{b}_{B_{2,\infty}^s} 2^{-ms},\end{equation} 
for all $m$, for some constant $K=K(s)$, where $\pi_m$ is the $L^2$--orthogonal projection onto $S_m$.  

\begin{remarks}
\emph{Uniform sup-norm convergence of the wavelet series.} The wavelet projections $\pi_m b$ converge to $b$ in supremum norm for any $b\in \Theta$, uniformly across $b\in\Theta$. That is, \begin{equation}\label{eqn:WaveletSeriesConverges}
\sup_{b\in\Theta}\norm{\pi_m b-b}_\infty \to 0 \quad \text{as}\quad m\to \infty.
\end{equation} This follows from Proposition 4.3.24 in \cite{Gine2016} since $K_0$ uniformly bounds $\norm{b}_{C_\per^1}$ for $b\in \Theta$.

\emph{Boundary regularity.} Functions in the periodic Besov space here denoted $B_{2,\infty}^s$ are $s$ regular at the boundary, in the sense that their weak derivatives of order $s$ are 1--periodic.

\emph{Alternative approximation spaces.} The key property we need for our approximation spaces is that \Cref{eqn:ApproxSpaces} and \Cref{eqn:WaveletSeriesConverges} hold. Of these, only the first is needed of our spaces for our main contraction result \Cref{thm:ConcretePriorContraction}. A corresponding inequality holds for many other function spaces if we replace $2^m$ by $D_m=\dim(S_m)$; for example, for $S_m$ the set of trigonometric polynomials of degree at most $m$, or (provided $s\leq s_{\max}$ for some given $s_{\max}\in \RR$) for $S_m$ generated by periodised Daubechies wavelets. Priors built using these other spaces will achieve the same posterior contraction rate. 
\end{remarks}

\section{Main contraction theorem}\label{sec:MainContractionTheorems}
Let $\Pi$ be a (prior) probability distribution on some $\sigma$--algebra $\Ss$ of subsets of $\Theta$. Given $b\sim \Pi$ assume that $(X_t: t\geq 0)$ 
follows the law $P_b$ as described in \Cref{sec:FrameworkAndAssumptions}. Write $p_b(\Delta,x,y)$ for the transition densities \[p_b(\Delta,x,y)\dy = P_b(X_\Delta \in \dy \mid X_0=x),\] and recall we use $p_0$ as shorthand for $p_{b_0}$. Assume that the mapping $(b,\Delta,x,y) \mapsto p_b(\Delta,x,y)$ is jointly measurable with respect to the $\sigma$--algebras $\Ss$ and $\Bb_\RR$, where $\Bb_\RR$ is the Borel $\sigma$--algebra on $\RR$. Then it can be shown by standard arguments that the Bayesian posterior distribution given the data is 
\[ b\mid X^{(n)} 
\sim \frac{\pi_b(X_0)\prod_{i=1}^n p_b(\Delta,X_{(i-1)\Delta},X_{i\Delta})\dPi(b)}{\int_\Theta \pi_b(X_0)\prod_{i=1}^n p_b(\Delta,X_{(i-1)\Delta},X_{i\Delta})\dPi(b)}\equiv\frac{p_b^{(n)}(X^{(n)})\dPi(b)}{\int_\Theta p_b^{(n)}(X^{(n)})\dPi(b)}, \]
where we introduce the shorthand $p_b^{(n)}(x^{(n)})=\pi_b(x_0)\prod_{i=1}^{n}p_b(\Delta,x_{(i-1)\Delta},x_{i\Delta})$ for the joint probability density of the data $(X_{0},\dots,X_{n\Delta})$.

A main result of this paper is the following. \Cref{thm:SievePriorContraction} is designed to apply to adaptive sieve priors, while \Cref{thm:sK0KnownContraction} is designed for use when the smoothness of the parameter $b$ is known. See \Cref{sec:ExplicitPriors} for explicit examples of these results in use and see \Cref{sec:MainContractionResultsProof} for the proof.
\begin{theorem}\label{thm:ConcretePriorContraction}
Consider data $X^{(n)}=(X_{k\Delta})_{0\leq k \leq n}$ sampled from a solution $X$ to \Cref{eqn:sde} under \Cref{assumptions:b,assumptions:sigma,assumption:Delta,assumption:invariant}. Let the true parameter be $b_0$. Assume the appropriate sets below are measurable with respect to the $\sigma$--algebra $\Ss$.
\begin{thmenum}
\item \label{thm:SievePriorContraction}
Let $\Pi$ be a sieve prior on $\Theta$, i.e.\ let $\Pi=\sum_{m=1}^\infty h(m)\Pi_m$, where $\Pi_m(S_m\cap\Theta)=1$, for $S_m$ a periodic Meyer-type wavelet space of resolution $m$ as described in \Cref{sec:ApproxSpaces}, and $h$ some probability mass function on $\NN$. Suppose we have, for all $\eps>0$ and $m\in \NN$, and for some constants $\zeta,\beta_1,\beta_2,B_1,B_2>0,$ 
\begin{enumerate}[(i)]
\item \label{sievepriorhyperparameter} $B_1 e^{-\beta_1 D_m} \leq h(m) \leq B_2 e^{-\beta_2 D_m}$,
\item \label{sievepriorsmallball} $\Pi_m (\braces{b\in S_m : \norm{b-\pi_{m}b_0}_2\leq \eps})\geq (\eps \zeta)^{D_m}$, 
\end{enumerate}
where $\pi_m$ is the $L^2$--orthogonal projection onto $S_m$ and $D_m=\dim(S_m)=2^m$.
Then for some constant $M=M(A_0,s,\Ii,L_0,\beta_1,\beta_2,B_1,B_2,\zeta)$ we have, for any $b_0\in \Theta_s(A_0)$,
\[\Pi\brackets[\Big]{\braces[big]{b\in \Theta : \norm{b-b_0}_2 \leq M (n\Delta)^{-s/(1+2s)}\log(n\Delta)^{1/2}} \mid X^{(n)}}\to 1\] in probability under the law $P_{b_0}$ of $X$. 

\item \label{thm:sK0KnownContraction}
Suppose now $b_0\in \Theta_s(A_0)$ where $s\geq 1$ and $A_0>0$ are both known. Let $j_n\in \NN$ be such that $D_{j_n}\sim (n\Delta)^{1/(1+2s)},$ i.e.\ for some positive constants $L_1,L_2$ and all $n\in \NN$ let $L_1(n\Delta)^{1/(1+2s)}\leq D_{j_n}\leq L_2(n\Delta)^{1/(1+2s)}$. 
Let $(\Pi^{(n)})_{n\in\NN}$ be a sequence of priors satisfying, for some constant $\zeta>0$ and for $\eps_n=(n\Delta)^{-s/(1+2s)}\log(n\Delta)^{1/2}$,
\begin{enumerate}[(I)] 
\item \label{smoothnesspriorsupport} $\Pi^{(n)}(\Theta_s(A_0)\cap \Theta)=1$ for all $n$,
\item \label{smoothnesspriorsmallball} $\Pi^{(n)} (\braces{b \in \Theta : \norm{\pi_{j_n} b-\pi_{j_n}b_0}_2\leq \eps_n})\geq (\eps_n \zeta)^{D_{j_n}}$.
\end{enumerate} 
Then we achieve the same rate of contraction; i.e.\ for some $M=M(A_0,s,\Ii,L_0,\zeta)$,
\[\Pi^{(n)}\brackets[\Big]{\braces[\big]{b\in \Theta : \norm{b-b_0}_2 \leq M (n\Delta)^{-s/(1+2s)}\log(n\Delta)^{1/2}} \mid X^{(n)}}\to 1\] in probability under the law $P_{b_0}$ of $X$.
\end{thmenum}
\end{theorem}

\begin{remark}
\emph{Optimality.} The minimax lower bounds of Hoffmann \cite{Hoffmann1999b} do not strictly apply because we have assumed $\sigma$ is given. Nevertheless, the minimax rate in this model should be $(n\Delta)^{-s/(1+2s)}$. This follows by adapting arguments for the continuous data case from Kutoyants \cite{Kutoyants2004} Section 4.5 to apply to the periodic model and observing that with high-frequency data we cannot outperform continuous data.

Since a contraction rate of $\eps_n$ guarantees the existence of an estimator converging to the true parameter at rate $\eps_n$ (for example, the centre of the smallest posterior ball of mass at least 1/2 -- see Theorem 8.7 in Ghosal \& van der Vaart \cite{Ghosal2017}) the rates attained in \Cref{thm:ConcretePriorContraction} are optimal, up to the log factors.
\end{remark}

\subsection{Explicit examples of priors}\label{sec:ExplicitPriors}

Our results guarantee that the following priors will exhibit posterior contraction.
Throughout this \namecref{sec:ExplicitPriors} we continue to adopt \Cref{assumptions:b,assumptions:sigma,assumption:Delta,assumption:invariant}, and for technical convenience, we add an extra assumption on $b_0$. Precisely, recalling that $\braces{\psi_{lk}}$ form a family of Meyer-type wavelets as in \Cref{sec:ApproxSpaces} and $\psi_{-1,0}$ denotes the constant function 1, we assume the following. 
\begin{assumption}\label{assumption:b0inB_infty1^1}
For a sequence $(\tau_l)_{l\geq -1}$ to be specified and a constant $B$, we assume
\begin{equation}\label{eqn:b0inB_infty1^1} b_0=\sum_{\substack{l\geq-1\\0\leq k<2^l}} \tau_l \beta_{lk}\psi_{lk}, \quad \text{with } \abs{\beta_{lk}}\leq B \text{ for all } l\geq -1 \text{ and all } 0\leq k<2^l. \end{equation}
\end{assumption}

The explicit priors for which we prove contraction will be random wavelet series priors.
Let $u_{lk}\iidsim q$, where $q$ is a density on $\RR$ satisfying \[q(x)\geq \zeta \text{ for } \abs{x}\leq B, \quad \text{and}\quad q(x)=0 \text{ for } \abs{x}>B+1,\] where $\zeta>0$ is a constant and $B>0$ is the constant from \Cref{assumption:b0inB_infty1^1}. For example one might choose $q$ to be the density of a $\operatorname{Unif}[0,B]$ random variable or a truncated Gaussian density.

We define a prior $\Pi_m$ on $S_m$ as the law associated to a random wavelet series 
\begin{equation}\label{eqn:Pi_m} b(x)=\sum_{\substack{-1\leq l< m\\ 0\leq k <2^l}} \tau_l u_{lk}\psi_{lk}(x), \qquad x\in[0,1], \end{equation} for $\tau_l$ as in \Cref{assumption:b0inB_infty1^1}.
We give three examples of priors built from these $\Pi_m$. 
\begin{example}[Basic sieve prior]\label{ex:SievePrior}
Let $\tau_{-1}=\tau_0=1$ and $\tau_l=2^{-3l/2}l^{-2}$ for $l\geq 1$.
Let $h$ be a probability distribution on $\NN$ as described in \Cref{thm:SievePriorContraction}, for example, $h(m)=\gamma e^{-2^m},$ where $\gamma$ is a normalising constant. 
Let $\Pi=\sum_{m=1}^\infty h(m)\Pi_m$ where $\Pi_m$ is as above.
\end{example}
\begin{proposition}\label{prop:BasicSievePrior}
The preceding prior meets the conditions of \Cref{thm:SievePriorContraction} for any $b_0$ satisfying \Cref{assumption:b0inB_infty1^1} with the same $\tau_l$ used to define the prior, and for an appropriate constant $K_0$. Thus, if also $b_0\in \Theta_s(A_0)$ for some constant $A_0$, 
$\Pi \brackets*{\braces{b \in \Theta : \norm{b-b_0}_2\leq M (n\Delta)^{-s/(1+2s)}\log(n\Delta)^{1/2}} \mid X^{(n)}} \to 1$ in $P_{b_0}$-probability, for some constant $M$.
\end{proposition}
The proof can be found in \Cref{sec:ExplicitPriorProofs}.
\begin{remark}\emph{Adaptive estimation.} If we assume $b_0\in \Theta_{s_\text{min}}(A_0)$, for some $s_\text{min}>3/2,$ \Cref{assumption:b0inB_infty1^1} automatically holds with $\tau_l$ as in \Cref{ex:SievePrior} for some constant $B=B(s_\text{min},A_0)$, as can be seen from the wavelet characterisation \Cref{eqn:B2inftywaveletcharacterisation}. Thus, in contrast to the low-frequency results of \cite{Nickl2017}, the above prior adapts to unknown $s$ in the range $s_{\text{min}}\leq s< \infty$.
\end{remark}

When $s> 1$ is known, we fix the rate of decay of wavelet coefficients to ensure a draw from the prior lies in $\Theta_s(A_0)$ by hand, rather than relying on the hyperparameter to choose the right resolution of wavelet space. We demonstrate with the following example. The proofs of \Cref{prop:KnownSmoothnessPrior,prop:InvariantDensityPrior}, also given in \Cref{sec:ExplicitPriorProofs}, mimic that of \Cref{prop:BasicSievePrior} but rely on \Cref{thm:sK0KnownContraction} in place of \Cref{thm:SievePriorContraction}. 

\begin{example}[Known smoothness prior]\label{ex:KnownSmoothnessPrior}
Let $\tau_{-1}=1$ and $\tau_l=2^{-l(s+1/2)}$ for $l\geq 0$.
Let $\bar{L}_n\in \NN\cup\braces{\infty}$. Define a sequence of priors $\Pi^{(n)}=\Pi_{\bar{L}_n}$ for $b$
(we can take $\bar{L}_n=\infty$ to have a genuine prior, but a sequence of priors will also work provided $\bar{L}_n\to \infty$ at a fast enough rate). 
\end{example}
\begin{proposition}\label{prop:KnownSmoothnessPrior}
Assume $\bar{L}_n/(n\Delta)^{1/(1+2s)}$ is bounded away from zero. 
Then for any $s>1$, the preceding sequence of priors meets the conditions of \Cref{thm:sK0KnownContraction} for any $b_0$ satisfying \Cref{assumption:b0inB_infty1^1} with the same $\tau_l$ used to define the prior, and for an appropriate constant $K_0$. Thus, for some constant $M$, $\Pi^{(n)} \brackets*{\braces{b \in \Theta : \norm{b-b_0}_2\leq M (n\Delta)^{-s/(1+2s)}\log(n\Delta)^{1/2}} \mid X^{(n)}} \to 1$ in $P_{b_0}$-probability.
\end{proposition}

\begin{remark}
\Cref{assumption:b0inB_infty1^1} with $\tau_l=2^{-l(s+1/2)}$ in fact forces $b_0\in B^s_{\infty,\infty} \subsetneq B^s_{2,\infty}$ with fixed norm bound. Restricting to this smaller set does not change the minimax rate, as can be seen from the fact that the functions by which Hoffmann perturbs in the lower bound proofs in \cite{Hoffmann1999b} lie in the smaller class addressed here. In principle, one could remove this assumption by taking $\tau_l=2^{-ls}$ and taking the prior $\Pi^{(n)}$ to be the law of $b\sim\Pi_{\bar{L}_n}$ conditional on $b\in\Theta_s(A_0)$.
\end{remark}

\begin{example}[Prior on the invariant density]\label{ex:InvariantDensityPrior}
In some applications it may be more natural to place a prior on the invariant density and only implicitly on the drift function. With minor adjustments, \Cref{thm:sK0KnownContraction} can still be applied to such priors. We outline the necessary adjustments.

\begin{enumerate}[(i)]
\item \label{item:densitypriorchange1} $b$ is not identifiable from $\pi_b$ and $\sigma^2$. We therefore introduce the identifiability constraint $I_b(1)=0$. We could fix $I_b(1)$ as any positive constant and reduce to the case $I_b(1)=0$ by a translation, so we choose $I_b(1)=0$ for simplicity (this assumption is standard in the periodic model, for example see van Waaij \& van Zanten \cite{vanWaaij2016}). With this restriction, we have $\pi_b(x)=\frac{e^{I_b(x)}}{G_b \sigma^2(x)}$ for a normalising constant $G_b$, so that $b=((\sigma^2)'+\sigma^2 (\log\pi_b)')/2$. 
\item \label{item:densitypriorchange2} In place of \Cref{assumption:b0inB_infty1^1}, we need a similar assumption but for $H_0:=\log\pi_{b_0}$. Precisely, we assume
\begin{equation} 
\label{eqn:H0condition} H_0=\sum_{\substack{l\geq-1\\0\leq k<2^l}} \tau_l h_{lk}\psi_{lk}, \quad \text{with } \abs{h_{lk}}\leq B \text{ for all } l\geq -1 \text{ and all } 0\leq k<2^l, \end{equation} for $\tau_{-1}=\tau_0=1$ and $\tau_l=2^{-l(s+3/2)}l^{-2}$ for $l\geq 1$, for some known constant $B$, and where $s\geq 1$ is assumed known.
\item \label{item:densitypriorchange3} Induce a prior on $b=((\sigma^2)'+\sigma^2 H')/2$ by putting the prior $\Pi^{(n)}=\Pi_{\bar{L}_n}$ on $H$, where $\bar{L}_n$ is as in \Cref{prop:KnownSmoothnessPrior}.
\item \label{item:densitypriorchange4} To ensure $b\in \Theta_s(A_0)$ we place further restrictions on $\sigma$; for example, we could assume $\sigma^2$ is smooth. More tightly, it is sufficient to assume (in addition to \Cref{assumptions:sigma}) that $\sigma^2\in \Theta_{s+1}(A_1)$ and $\norm{\sigma^2}_{C^s_\per}\leq A_1$, where $C^s_\per$ is the H{\"o}lder norm, for some $A_1>0$. These conditions on $\sigma$ can be bypassed with a more careful statement of \Cref{thm:sK0KnownContraction} and a more careful treatment of the bias.
\end{enumerate} 
\begin{proposition}\label{prop:InvariantDensityPrior}
Make changes \Cref{item:densitypriorchange1,item:densitypriorchange2,item:densitypriorchange3,item:densitypriorchange4} as listed. Then, the obtained sequence of priors meets the conditions of \Cref{thm:sK0KnownContraction} for an appropriate constant $K_0$, hence for some constant $M$ we have $\Pi^{(n)} \brackets*{\braces{b \in \Theta : \norm{b-b_0}_2\leq M (n\Delta)^{-s/(1+2s)}\log(n\Delta)^{1/2}} \mid X^{(n)}} \to 1$ in $P_{b_0}$-probability.
\end{proposition}
\end{example}

\begin{remarks}
\emph{Minimax rates.} The assumption \Cref{eqn:H0condition} restricts $b_0$ beyond simply lying in $\Theta_s(A_0)$. 
As with Nickl \& S{\"o}hl \cite{Nickl2017} Remark 5, this further restriction does not change the minimax rates, except for a log factor induced by the weights $l^{-2}$. 

\emph{Adaptation to sampling regime.} 
The prior of \Cref{prop:InvariantDensityPrior} is the same as the prior on $b$ in \cite{Nickl2017}. However, since here we assume $\sigma$ is given while in \cite{Nickl2017} it is an unknown parameter, the results of \cite{Nickl2017} do not immediately yield contraction of this prior at a near-minimax rate in the low-frequency setting. In particular, when $\sigma$ is known the minimax rate for estimating $b$ with low-frequency data is $n^{-s/(2s+3)}$ (for example see S{\"o}hl \& Trabs \cite{Sohl2016}), rather than the slower rate $n^{-s/(2s+5)}$ attained in Gobet--Hoffmann--Reiss \cite{Gobet2004} when $\sigma$ is unknown (this improvement is possible because one bypasses the delicate interweaving of the problems of estimating $b$ and $\sigma$ with low-frequency data). Nevertheless, the prior of \Cref{prop:InvariantDensityPrior} will indeed exhibit near-minimax contraction also in the low-frequency setting. An outline of the proof is as follows. The small ball results of \cite{Nickl2017} still apply, with minor changes to the periodic model used here in place of their reflected diffusion, so it is enough to exhibit tests of the true parameter against suitably separated alternatives. The identification $b=((\sigma^2)'+\sigma^2 (\log\pi_b)')/2$ means one can work with the invariant density rather than directly with the drift. Finally one shows the estimator from \cite{Sohl2016} exhibits sufficiently good concentration properties (alternatively, one could use general results for Markov chains from Ghosal \& van der Vaart \cite{Ghosal2007}).

It remains an interesting open problem to simultaneously estimate $b$ and $\sigma$ with a method which adapts to the sampling regime. Extending the proofs of this paper to the case where $\sigma$ is unknown would show that the Bayesian method fulfils this goal. The key difficulty in making this extension arises in the small ball section (\Cref{sec:SmallBallProbs}), because Girsanov's Theorem does not apply to diffusions with different diffusion coefficients.

\emph{Intermediate sampling regime.}
Strictly speaking, we only demonstrate robustness to the sampling regime in the extreme cases where $\Delta>0$ is fixed or where $n\Delta^2\to 0$. The author is not aware of any papers addressing the intermediate regime (where $\Delta$ tends to $0$ at a slower rate than $n^{-1/2}$) for a nonparametric model: the minimax rates do not even appear in the literature. Since the Bayesian method adapts to the extreme regimes, one expects that it attains the correct rates in this intermediate regime (up to log factors). However, the proof would require substantial extra work, primarily in exhibiting an estimator with good concentration properties in this regime. Kessler's work on the intermediate regime in the parametric case \cite{Kessler1997} would be a natural starting point for exploring this regime in the nonparametric setting.
\end{remarks}

\section{Construction of tests}\label{sec:Concentration}
In this section we construct the tests needed to apply the general contraction rate theory from Ghosal--Ghosh--van der Vaart \cite{GGV2000}. The main result of this section is the following. Recall that $S_m$ is a periodic Meyer-type wavelet space of resolution $m$ as described in \Cref{sec:ApproxSpaces}, $\pi_m$ is the $L^2$--orthogonal projection onto $S_m$ and $D_m=\dim(S_m)=2^m$.
\begin{lemma}\label{lem:ExistenceOfTests} Consider data $X^{(n)}=(X_{k\Delta})_{0\leq k \leq n}$ sampled from a solution $X$ to \Cref{eqn:sde} under \Cref{assumptions:b,assumptions:sigma,assumption:Delta,assumption:invariant}. Let $\eps_n\to 0$ be a sequence of positive numbers and let $l_n\to \infty$ be a sequence of positive integers such that $n\Delta\eps_n^2/\log(n\Delta)\to \infty$ and, for some constant $L$ and all $n$, $D_{l_n}\leq L n\Delta \eps_n^2$. Let $\Theta_n\subseteq \braces{b\in \Theta \st \norm{\pi_{l_n}b-b}_2\leq \eps_n}$ contain $b_0$.

Then for any $D>0$, there is an $M=M(\Ii,L_0,D,L)>0$ for which there exist tests $\psi_n$ (i.e.\ $\braces{0,1}$--valued functions of the data) such that, for all $n$ sufficiently large,
\[ \max \brackets[\Big]{ E_{b_0} \psi_n (X^{(n)}), \sup \braces[\big]{ E_b \sqbrackets{1- \psi_n(X^{(n)})} : b\in\Theta_n,  \norm{b-b_0}_2>M\eps_n} }
\leq e^{-Dn\Delta\eps_n^2}. \]
\end{lemma}

The proof is given in \Cref{sec:ComteConcentration} and is a straightforward consequence of our constructing an estimator with appropriate concentration properties. First, we introduce some general concentration results we will need. 
\subsection{General concentration results} 
We will use three forms of concentration results as building blocks for our theorems. The first comes from viewing the data $(X_{j\Delta})_{0\leq j\leq n}$ as a Markov chain and applying Markov chain concentration results; these results are similar to those used in Nickl \& S{\"o}hl \cite{Nickl2017} for the low-frequency case, but here we need to track the dependence of constants on $\Delta$. 
The second form are useful only in the high-frequency case because they use a quantitative form of H{\"o}lder continuity for diffusion processes.
An inequality of the third form, based on martingale properties, is introduced only where needed (in \Cref{lem:nu}).
\subsubsection{Markov chain concentration results applied to diffusions}
Our main concentration result arising from the Markov structure is the following. We denote by $\norm{\cdot}_\mu$ the  $L^2_\mu([0,1])$--norm, $\norm{f}_\mu^2=E_\mu[f^2]=\int_0^1 f(x)^2\dmu(x)$.

\begin{theorem}\label{thm:concentration-for-diffusions}  There exists a constant $\kappa=\kappa(\Ii)$ such that, for all $n$ sufficiently large and all bounded 1--periodic functions $f:\RR\to \RR$,
\begin{equation} \label{eqn:DiffusionConcentrationMinForm}P_b \brackets*{ \abs[\Big]{\sum_{k=1}^n f(X_{k\Delta})-E_\mu [f]}\geq t } \leq 2 \exp \brackets*{ -\frac{1}{\kappa} \Delta \min \brackets*{\frac{t^2}{n\norm{f}_\mu^2},\frac{t}{\norm{f}_\infty} }},\end{equation}
or equivalently 
\begin{equation}\label{eqn:DiffusionConcentrationMaxForm}P_{b}\brackets*{\abs[\Big]{\sum_{j=1}^{n} f(X_{j\Delta})-E_\mu [f]}\geq \max(\sqrt{\kappa v^2x},\kappa ux)} \leq 2e^{-x},\end{equation}
where $v^2= n \Delta^{-1} \norm{f}_\mu^2$ and $u= \Delta^{-1}\norm{f}_\infty$.

Further, if $\Ff$ is a space of such functions indexed by some (subset of a) $d$--dimensional vector space, then for $V^2=\sup_{f\in\Ff} v^2$ and $U=\sup_{f\in\Ff} u$, we also have 
\begin{equation}\label{eqn:DiffusionConcentrationSupForm}P_{b}\brackets*{ \sup_{f\in \Ff} \abs[\Big]{\sum_{j=1}^{n} f(X_{j\Delta})-E_\mu[f]}\geq \tilde{\kappa}\max\braces*{\sqrt{V^2(d+x)},U(d+x)}}\leq 4e^{-x}.\end{equation}
for some constant $\tilde{\kappa}=\tilde{\kappa}(\Ii)$.
\end{theorem}

The proof is an application of the following abstract result for Markov chains. 
 
\begin{theorem}[Paulin \cite{Paulin2015}, Proposition 3.4 and Theorem 3.4]\label{thm:Paulin}
Let $M_1,\dots,M_n$ be a time-homogeneous Markov chain taking values in $S$ with transition kernel $P(x,\dy)$ and invariant density $\pi$. Suppose $M$ is uniformly ergodic, i.e.\ $\sup_{x\in S} \norm{P^n(x,\cdot)-\pi}_{TV}\leq K\rho^n$ for some constants $K<\infty$, $\rho<1$, where $P^n(x,\cdot)$ is the $n-$step transition kernel and $\norm{\cdot}_{TV}$ is the total variation norm for signed measures. Write $t_\text{mix}=\min\braces{n\geq 0 : \sup_{x\in S}\norm{P^n(x,\cdot)-\pi}_{TV}<1/4}.$ 
Suppose $M_1\sim \pi$ and $f: S\to \RR$ is bounded. Let $V_f=\Var[f(M_1)]$, let $C=\norm{f-E[f(M_1)]}_\infty$. Then
\[ P \brackets*{ \abs{\sum_{i=1}^n f(M_i)-E [f(M_i)]}\geq t}\leq 2\exp\brackets*{ \frac{-t^2}{2t_{\text{mix}}(8(n+2t_{\text{mix}})V_f+20tC)}}.\]
\end{theorem}

\begin{proof}[Proof of \Cref{thm:concentration-for-diffusions}]

Since $f$ is assumed periodic we see that $f(X_{k\Delta})=f(\dot{X}_{k\Delta}),$ where we recall $\dot{X}=X\mod 1$. Denote by $\dot{p}_b(t,x,y)$ the transition densities of $\dot{X}$, i.e.\
$ \dot{p}_b(t,x,y) = \sum_{j\in \ZZ} p_b(t,x,y+j)$ (see the proof of Proposition 9 in Nickl \& S{\"o}hl \cite{Nickl2017} for an argument that the sum converges).  
Theorem 2.6 in Bhattacharya et al.\ \cite{Bhattacharya1999} tells us that if $\dot{X}_0$ has a density $\eta_0$ on $[0,1]$, then $\dot{X}_t$ has a density $\eta_t$ satisfying 
\[ \norm{\eta_t -\pi_b}_{TV} \leq \frac{1}{2}\norm{\eta_0/\pi_b - 1}_{TV} \exp\brackets[\Big]{-\frac{1}{2M_b} t},\] where  $M_b:=\sup_{z\in[0,1]}\braces[\Big]{(\sigma^2(z)\pi_b(z))^{-1} \int_0^z \pi_b(x)\dx \int_z^1 \pi_b(y)\dy}.$ 
We can regularise to extend the result so that it also applies when the initial distribution of $\dot{X}$ is a point mass: if $\dot{X}_0=x$ then $\dot{X}_1$ has density $\dot{p}_b(1,x,\cdot),$ hence the result applies to show \[\norm{\dot{p}_b(t,x,\cdot)-\pi_b}_{TV} \leq \frac{1}{2} \norm{\dot{p}_b(1,x,\cdot)/\pi_b - 1}_{TV} \exp\brackets[\Big]{-\frac{1}{2M_b}(t-1)}.\] Moreover, note $\norm{\dot{p}_b(1,x,\cdot)/\pi_b-1}_{TV}\leq \pi_L^{-1}\norm{\dot{p}_b(1,x,\cdot)-\pi_b}_{TV}\leq \pi_L^{-1}.$
Also note we can upper bound $M_b$ by a constant $M=M(\Ii)$: precisely, we can take $M=\sigma_L^{-2}\pi_L^{-1}\pi_U^2$. 

Thus, we see that for $t\geq 1$, we have 
\[\norm{\dot{p}_b(t,x,\cdot)-\pi_b}_{TV} \leq K \exp\brackets[\Big]{-\frac{1}{2M}t}\] for some constant $K=K(\Ii)$, uniformly across $x\in[0,1]$.
It follows that, for each fixed $\Delta$, the discrete time Markov chain $(\dot{X}_{k\Delta})_{k\geq 0}$ is uniformly ergodic with mixing time $t_{\text{mix}}\leq 1+2M\log(4K)\Delta^{-1}\leq K'\Delta^{-1}$ for some constant $K'$. 
\Cref{thm:Paulin} applies to tell us 

\[ P \brackets*{ \abs{\sum_{i=1}^n f(X_{k\Delta})-E _\mu[f]}\geq t}\leq 2\exp\brackets*{ -\frac{t^2}{2K'\Delta^{-1}(8(n+2K'\Delta^{-1})V_f+20tC)}}.\]
Since $n\Delta\to \infty$ by assumption, we see $8(n+2K'\Delta^{-1})\leq K''n$ for some constant $K''$. Using the bound $2/(a+b)\geq \min\brackets{1/a,1/b}$ for $a,b>0$ and upper bounding the centred moments $V_f$ and $C$ by the uncentred moments $\norm{f}_\mu^2$ and $\norm{f}_\infty$, we deduce \eqref{eqn:DiffusionConcentrationMinForm}.

The result \Cref{eqn:DiffusionConcentrationMaxForm} is obtained by a change of variables. For the supremum result \Cref{eqn:DiffusionConcentrationSupForm}, we use a standard chaining argument, eg.\ as in Baraud \cite{Baraud2010} Theorem 2.1, where we use \Cref{eqn:DiffusionConcentrationMaxForm} in place of Baraud's Assumption 2.1, noting that Baraud only uses Assumption 2.1 to prove an expression mirroring \Cref{eqn:DiffusionConcentrationMaxForm}, and the rest of the proof follows through exactly. Precisely, following the proof, we can take $\tilde{\kappa}=36\kappa$. 
\end{proof}
\begin{remark}
The proof simplifies if we restrict $\Theta$ to only those $b$ satisfying $I_b(1)=0$. In this case, the invariant density (upon changing normalising constant to some $G_b$) reduces to the more familiar form $\pi_b(x)=(G_b\sigma^2(x))^{-1}e^{I_b(x)}$. The diffusion is reversible in this case, and we can use Theorem 3.3 from \cite{Paulin2015} instead of Theorem 3.4 to attain the same results but with better constants. 
\end{remark}

\subsubsection{H{\"o}lder continuity properties of diffusions}\label{sec:HolderProperties}
Define \[ w_m(\delta) = \delta^{1/2} \brackets{\brackets{\log \delta^{-1}}^{1/2}+\log(m)^{1/2}}, \qquad \delta\in (0,1] \] for $m\geq 1$, and write $w_m(\delta):=w_1(\delta)$ for $m<1$. The key result of this section is the following.
\begin{lemma} \label{lem:X_holder}
Let $X$ solve the scalar diffusion equation \Cref{eqn:sde}, and grant \Cref{assumptions:b,assumptions:sigma}. Then there exist positive constants $\lambda$, $C$ and $\tau$, all depending on $\Ii$ only, such that for any $u>C\max\brackets{\log(m),1}^{1/2}$ and for any initial value $x$, \[ P_{b}^{(x)}\brackets*{ \sup_{\substack{s,t\in [0,m],\\t\not=s, \abs{t-s}\leq \tau}} \brackets*{\frac{\abs{X_t-X_s}}{w_m(\abs{t-s})}} > u}\leq 2e^{-\lambda u^2}.\]  
\end{lemma}

\begin{remarks}
\begin{enumerate}[i.]
\item
We will need to control all increments $X_{(j+1)\Delta}-X_{j\Delta}$ simultaneously, hence we include the parameter $m$, which we will take to be the time horizon $n\Delta$ when applying this result. Simply controlling over $[0,1]$ and using a union bound does not give sharp enough results.
\item The \namecref{lem:X_holder} applies for any distribution of $X_0$, not only point masses, by an application of the tower law.
\end{enumerate}
\end{remarks}
The modulus of continuity $w_m$ matches that of Brownian motion, and indeed the proof, given in \Cref{sec:HolderPropertiesProofs}, is to reduce to the corresponding result for Brownian motion. First, by applying the scale function one transforms $X$ into a local martingale, reducing \Cref{lem:X_holder} to the following result, also useful in its own right.

\begin{lemma}\label{lem:Y_holder}
Let $Y$ be a local martingale with quadratic variation satisfying $\abs{\qv{Y}_t-\qv{Y}_s} \leq A \abs{t-s}$ for a constant $A\geq 1$. 
Then there exist positive constants $\lambda=\lambda(A)$ and $C=C(A)$ such that for any $u>C\max\brackets{\log(m),1}^{1/2}$, 
\[ \Pr\brackets*{ \sup_{\substack{s,t\in [0,m], s\not=t,\\ \abs{t-s}\leq A^{-1}e^{-2}}} \brackets*{\frac{\abs{Y_t-Y_s}}{w_m(\abs{t-s})}} > u}\leq 2e^{-\lambda u^2}.\] 
 
In particular the result applies when $Y$ is a solution to $\dY_t=\tilde{\sigma}(Y_t)\dW_t,$ provided $\norm{\tilde{\sigma}^2}_\infty \leq A.$
 \end{lemma}
\Cref{lem:Y_holder} follows from the corresponding result for Brownian motion by a time change (i.e.\ the (Dambis--)Dubins-Schwarz Theorem). 
It is well known that Brownian motion has modulus of continuity $\delta^{1/2}\brackets{\log\delta^{-1}}^{1/2}$
in the sense that there almost surely exists a constant $C>0$ such that
$\abs{B_t-B_s}\leq C\abs{t-s}^{1/2}\brackets{\log\brackets{\abs{t-s}^{-1}}}^{1/2},$ for all  $t,s\in[0,1]$ sufficiently close, but \Cref{lem:X_holder,lem:Y_holder} depend on the following quantitative version of this statement, proved using Gaussian process techniques. The proofs of \Cref{lem:Y_holder,lem:B_holder} are given in \Cref{sec:HolderPropertiesProofs}.

\begin{lemma}\label{lem:B_holder}
	Let $B$ be a standard Brownian motion on $[0,m]$. 
	There are postive (universal) constants $\lambda$ and $C$ such that for $u>C\max(\log(m),1)^{1/2}$, \[ \Pr\brackets*{ \sup_{\substack{s,t\in [0,m],\\ s\not=t, \abs{t-s}\leq e^{-2}}} \brackets*{\frac{\abs{B_t-B_s}}{w_m(\abs{t-s})}} > u}\leq 2e^{-\lambda u^2}.\] 
\end{lemma}

\subsection{Concentration of a drift estimator}\label{sec:ComteConcentration}

\subsubsection{Defining the estimator}
We adapt an estimator introduced in Comte et al.\ \cite{Comte2007}. 
The estimator is constructed by considering drift estimation as a regression-type problem. Specifically, defining \[
Z_{k\Delta}=\frac{1}{\Delta}\int_{k\Delta}^{(k+1)\Delta} \sigma(X_s)\dW_s, \qquad
R_{k\Delta}=\frac{1}{\Delta}\int_{k\Delta}^{(k+1)\Delta} (b(X_s)-b(X_{k\Delta}))\ds,
\]
we can write
\[ \frac{X_{(k+1)\Delta}-X_{k\Delta}}{\Delta}=b(X_{k\Delta})+Z_{k\Delta}+R_{k\Delta}.\] Note $R_{k\Delta}$ is a discretization error which vanishes as $\Delta\to 0$ and $Z_{k\Delta}$ takes on the role of noise. 
We define the \emph{empirical norm} and the related \emph{empirical loss function} \[\norm{u}_n=\frac{1}{n}\sum_{k=1}^n u(X_{k\Delta})^2, \quad 
\gamma_n(u)=\frac{1}{n} \sum_{k=1}^n [\Delta^{-1}(X_{(k+1)\Delta}-X_{k\Delta})-u(X_{k\Delta})]^2, \quad u:[0,1]\to \RR.\]
In both we leave out the $k=0$ term for notational convenience.

Recalling that $S_m$ is a Meyer-type wavelet space as described in \Cref{sec:ApproxSpaces} and $K_0$ is an upper bound for the $C_\per^1$--norm of any $b\in \Theta$, for $l_n$ to be chosen we define $\tilde{b}_n$ as a solution to the minimisation problem
\[\tilde{b}_n\in \argmin_{u\in \tilde{S}_{l_n}} \gamma_n(u), \qquad \tilde{S}_m :=\braces{u\in S_m : \norm{u}_\infty \leq K_0+1},\] where we choose arbitrarily among minimisers if there is no unique minimiser.\footnote{It is typical that we do not have uniqueness, since if $u$ is a minimiser of $\gamma_n$, then so is any $\tilde{u}\in\tilde{S}_{l_n}$ such that $\tilde{u}(X_{k\Delta})=u(X_{k\Delta})$ for $1\leq k\leq n$.}

\subsubsection{Main concentration result}
For the estimator defined above we will prove the following concentration inequality. \begin{theorem}\label{thm:bhat_concentrates} 
Consider data $X^{(n)}=(X_{k\Delta})_{0\leq k \leq n}$ sampled from a solution $X$ to \Cref{eqn:sde} under \Cref{assumptions:b,assumptions:sigma,assumption:Delta,assumption:invariant}. Let $\eps_n\to 0$ be a sequence of positive numbers and let $l_n\to \infty$ be a sequence of positive integers such that $n\Delta\eps_n^2/\log(n\Delta)\to \infty$ and, for some constant $L$ and all $n$, $D_{l_n}\leq L n\Delta \eps_n^2$. For these $l_n$, let $\tilde{b}_n$ be defined as above and let $\Theta_n\subseteq \braces{b\in \Theta \st \norm{\pi_{l_n}b-b}_2\leq \eps_n}$ contain $b_0$, where $\pi_{l_n}$ is the $L^2-$orthogonal projection onto $S_{l_n}$.

Then for any $D>0$ there is a $C=C(\Ii,L_0,D,L)>0$ such that, uniformly across $b\in \Theta_n$, 
\[ P_{b} \brackets*{\norm{\tilde{b}_n -b}_2>C\eps_n}\leq e^{-Dn\Delta \eps_n^2},\] for all $n$ sufficiently large.
\end{theorem}

\begin{remark}
Previous proofs of Bayesian contraction rates using the concentration of estimators approach (see \cite{Gine2011a},\cite{Nickl2017},\cite{Ray2013}) have used duality arguments, i.e.\ the fact that $\norm{f}_2=\sup_{v : \norm{v}_2=1} \ip{f,v}$, to demonstrate that the linear estimators considered satisfy a concentration inequality of the desired form. A key insight of this paper is that for the model we consider we can achieve the required concentration using the above \emph{minimum contrast} estimator (see Birg{\'e} \& Massart \cite{Birge1998}), for which we need techniques which differ substantially from duality arguments. 
\end{remark}
Before proceeding to the proof, we demonstrate how this can be used to prove the existence of tests of $b_0$ against suitably separated alternatives. 
\begin{proof}[Proof of \Cref{lem:ExistenceOfTests}]
Let $\tilde{b}_n$ be the estimator outlined above and let $D>0$. Let $C=C(\Ii,L_0,D,L)$ be as in \Cref{thm:bhat_concentrates} and let $M=2C$. It's not hard to see that $\psi_n=\II\braces{\norm{\tilde{b}_n-b}_2>C\eps_N}$ is a test with the desired properties.
\end{proof}

\begin{proof}[Proof of \Cref{thm:bhat_concentrates}]
It is enough to show that, uniformly across $b\in\Theta_n$, for any $D>0$ there is a $C>0$ such $ P_{b} \brackets*{\norm{\tilde{b}_n -b}_2>C\eps_n}\leq 14 e^{-Dn\Delta\eps_n^2},$ because by initially considering a $D'>D$ and finding the corresponding $C'$, we can eliminate the factor of $14$ in front of the exponential. 

The proof is structured as follows. Our assumptions ensure that the $L^2$-- and $L^2(\mu)$--norms are equivalent. We further show that the $L^2(\mu)$--norm is equivalent to the empirical norm $\norm{\cdot}_n$ on an event of sufficiently high probability. Finally, the definition of the estimator will allow us to control the empirical distance $\norm{\tilde{b}_n-b}_n$. 

To this end, write $\tilde{t}_n=(\tilde{b}_n-\pi_{l_n}b)\norm{\tilde{b}_n-\pi_{l_n}b}_\mu^{-1}$ (defining $\tilde{t}_n=0$ if $\tilde{b}_n=\pi_{l_n}b$) and introduce the following set and events: 
\begin{align*} 
I_n&=\braces*{t\in S_{l_n} \st \norm{t}_\mu=1, \norm{t}_\infty \leq C_1 \eps_n^{-1}}, \\ 
\Aa_n &= \braces*{\tilde{t}_n\in I_n}\cup\braces{\tilde{t}_n=0},\\
\Omega_n&=\braces*{\abs*{\norm{t}_n^2-1}\leq \frac{1}{2}, \:\forall t\in I_n}, 
\end{align*} 
where the constant $C_1$ is to be chosen. Then we can decompose 
\[ P_b\brackets[\big]{\norm{\tilde{b}_n-b}_2>C\eps_n}\leq  P_b\brackets[\big]{\norm{\tilde{b}_n-b}_2\II_{\Aa_n^c}>C\eps_n}+P_b\brackets[\big]{\Omega_n^c} +P_b(\brackets[\big]{\norm{\tilde{b}_n-b}_2\II_{\Aa_n\cap \Omega_n}>C\eps_n}.\] 
Thus, we will have proved the theorem once we have completed the following:
\begin{enumerate}
\item Show the theorem holds (deterministically) on $\Aa_n^c$, for a large enough constant $C$. 
\item Show that $P_{b}(\Omega_n^c) \leq 4e^{-Dn\Delta\eps_n^2}$ for a suitable choice of $C_1$. 
\item Show that, for any $D$, we can choose a $C$ such that  $P_b\brackets[\big]{\norm{\tilde{b}_n-b}_2\II_{\Aa_n\cap\Omega_n}>C\eps_n}\leq 10e^{-Dn\Delta\eps_n^2}$. 
\end{enumerate}

\paragraph{Step 1:}
Intuitively we reason thus. The event $\Aa_n^c$ can only occur if the $L^2(\mu)$--norm of $\tilde{b}_n-\pi_{l_n}b$ is small compared to the $L^\infty$--norm. Since we have assumed a uniform supremum bound on functions $b\in\Theta$, in fact $\Aa_n$ holds unless the $L^2(\mu)$--norm is small in absolute terms. But if $\norm{\tilde{b}_n-\pi_{l_n}b}_\mu$ is small, then so is $\norm{\tilde{b}_n-b}_2$. We formalise this reasoning now.

For a constant $C_2$ to be chosen, define \[\Aa_n'=\braces{\norm{\tilde{b}_n-\pi_{l_n}b}_\mu> C_2 \eps_n}.\] On $\Aa_n'$ we have $\norm{\tilde{t}_n}_\infty \leq (\norm{\tilde{b}_n}_\infty+\norm{\pi_{l_n}b}_\infty)C_2^{-1}\eps_n^{-1}.$ Note $\norm{\tilde{b}_n}_\infty\leq K_0+1$ by definition. Since, for $n$ large enough, $\norm{\pi_{l_n}b-b}_\infty \leq 1$ uniformly across $b\in \Theta_n\subseteq \Theta$ by \Cref{eqn:WaveletSeriesConverges} so that $\norm{\pi_{l_n}b}_\infty\leq\norm{b}_\infty+1\leq K_0+1$, we deduce that on $\Aa_n'$, $\norm{\tilde{t}_n}_\infty \leq (2K_0+2)C_2^{-1}\eps_n^{-1}$. Since also $\norm{\tilde{t}_n}_\mu=1$ (or $\tilde{t}_n=0$) by construction, we deduce $\Aa_n'\subseteq \Aa_n$ if $C_2\geq C_1^{-1}(2K_0+2)$.

Then on $(\Aa_n')^c\supseteq \Aa_n^c$ we find, using that $b\in\Theta_n$ and using $\norm{\cdot}_2\leq \pi_L^{-1/2}\norm{\cdot}_\mu$, \[\norm{\tilde{b}_n-b}_2\leq \norm{\tilde{b}_n-\pi_{l_n}b}_2+\norm{\pi_{l_n}b-b}_2 \leq (C_2\pi_L^{-1/2}+1)\eps_n.\] So on $\Aa_n^c$, we have $\norm{\tilde{b}_n-b}_2\leq C\eps_n$ deterministically for any $C\geq C_2\pi_L^{-1/2}+1$. That is, for $C$ large enough (depending on $C_1$ and $\Ii$), $P_b\brackets[\big]{\norm{\tilde{b}_n-b}_2\II_{\Aa_n^c}>C\eps_n}=0$.

\paragraph{Step 2:}
We show that for $n$ sufficiently large, and $C_1=C_1(\Ii,D,L)$ sufficiently small, $P_{b}(\Omega_n^c)\leq 4e^{-Dn\Delta\eps_n^2}.$

For $t\in I_n$ we have $\abs[\Big]{\norm{t}_n^2-1}=n^{-1}\abs[\Big]{\sum_{k=1}^n t^2(X_{k\Delta})-E_\mu[t^2]}.$ Thus \Cref{thm:concentration-for-diffusions} can be applied to $\Omega_n^c=\braces*{\sup_{t\in I_n} n^{-1}\abs[\Big]{\sum_{k=1}^n t^2(X_{k\Delta})-E_\mu[t^2]}>1/2}.$
Each $t\in I_n$ has $\norm{t^2}_\infty \leq C_1^2 \eps_n^{-2}$ and 
$\norm{t^2}_\mu^2= E_\mu [t^4] \leq \norm{t^2}_\infty \norm{t}_\mu^2\leq C_1^2 \eps_n^{-2}.$
Since the indexing set $I_n$ lies in a vector space of dimension $D_{l_n}$, we apply the \namecref{thm:concentration-for-diffusions} with $x=Dn\Delta\eps_n^2$ to see
\[P_{b}\brackets*{\sup_{t\in I_n}\abs*{\sum_{k=1}^n t^2(X_{k\Delta})-E_\mu[t^2]}\geq 36\max\braces{A,B}}\leq 4 e^{-Dn\Delta\eps_n^2}.\]
where $A=\sqrt{\tilde{\kappa} C_1^2 n\Delta^{-1} \eps_n^{-2} (Dn\Delta\eps_n^2+D_{l_n})}$ and $B=\tilde{\kappa} C_1^2 \Delta^{-1}\eps_n^{-2}(Dn\Delta\eps_n^2+D_{l_n})$, for some constant $\tilde{\kappa}=\tilde{\kappa}(\Ii)$.
Provided we can choose $C_1$ so that $36\max\braces{A/n,B/n}\leq 1/2$ the result is proved. Such a choice for $C_1$ can be made as we have assumed $D_{l_n}\leq L n\Delta \eps_n^2$.

\paragraph{Step 3:}
Since $b\in \Theta_n$ and $\pi_{l_n}$ is $L^2$-orthogonal projection, we have $\norm{\tilde{b}_n-b}_2^2\leq \norm{\tilde{b}_n-\pi_{l_n}b}_2^2+\eps_n^2$.  Recall that $\norm{\cdot}_2\leq \pi_L^{-1/2}\norm{\cdot}_\mu$ and note that on $\Aa_n\cap \Omega_n$, we further have  $\frac{1}{2}\norm{\tilde{b}_n-\pi_{l_n}b}_\mu^2\leq\norm{\tilde{b}_n-\pi_{l_n}b}_n^2.$ 

Since also $\norm{\tilde{b}_n-\pi_{l_n}b}_n^2\leq 2(\norm{\pi_{l_n}b-b}_n^2+\norm{\tilde{b}_n-b}_n^2)$ we deduce that
\[\norm{\tilde{b}_n-b}_2^2\II_{\Aa_n\cap \Omega_n}\leq \frac{1}{\pi_L} \brackets*{4\norm{\pi_{l_n}b-b}_n^2+ 4\norm{\tilde{b}_n-b}_n^2\II_{\Aa_n\cap \Omega_n}}+\eps_n^2,\]
where we have dropped indicator functions from terms on the right except where we will need them later.
Thus, using a union bound, \[P_b(\norm{\tilde{b}_n-b}_2\II_{\Aa_n\cap \Omega_n}>C\eps_n)\leq P_b\brackets[\big]{\norm{\pi_{l_n}b-b}_n^2>C'\eps_n^2}+ P_b\brackets[\big]{\norm{\tilde{b}_n-b}_n^2\II_{\Aa_n\cap \Omega_n}>C'\eps_n^2},\] for some constant $C'$ (precisely we can take $C'=\pi_L (C^2-1)/8$). It remains to show that both probabilities on the right are exponentially small.

\paragraph{Bounding $P_{b}\brackets*{\norm{\pi_{l_n}b-b}_n > C \eps_n}$:}
We show that for any $D>0$ there is a constant $C$ such that $P_{b}\brackets*{\norm{\pi_{l_n}b-b}_n > C \eps_n}\leq 2e^{-Dn\Delta\eps_n^2},$ for all $n$ sufficiently large.
Since $E_{b} \norm{g}_n^2=\norm{g}_\mu^2$ for any \mbox{1--periodic} deterministic function $g$ and $\norm{\pi_{l_n} b-b}_\mu^2 \leq \pi_U\norm{\pi_{l_n}b-b}_2^2\leq \pi_U \eps_n^2$ for $b\in \Theta_n$, it is enough to show that 
\begin{equation} \label{eqn:normnConcentrates'}P_{b}\brackets*{\abs[\big]{\norm{\pi_{l_n}b-b}_n^2-E_b\norm{\pi_{l_n}b-b}_n^2} > C \eps_n^2}\leq 2e^{-Dn\Delta\eps_n^2}\end{equation} 
for some different $C$. As in Step 2, we apply \Cref{thm:concentration-for-diffusions}, but now working with the single function $(\pi_{l_n}b-\nobreak b)^2$. For large enough $n$ we have the bounds $\norm{\pi_{l_n}b-b}_\infty\leq 1$ (derived from \Cref{eqn:WaveletSeriesConverges}), and $\norm{(\pi_{l_n}b-\nobreak b)^2}_\mu \leq$ $\norm{\pi_{l_n}b-b}_\infty\norm{\pi_{l_n}b-b}_\mu\leq \pi_U^{1/2} \eps_n$ (because $b\in\Theta_n$) and  so applying the \namecref{thm:concentration-for-diffusions} with $x=Dn\Delta\eps_n^2$ gives
\[ P_{b} \brackets*{\abs*{\sum_{k=1}^n\sqbrackets*{ (\pi_{l_n}b-b)^2(X_{k\Delta}) - \norm{\pi_{l_n}b-b}_\mu^2}}\geq \max\braces{a,b}} \leq 2 e^{-Dn\Delta\eps_n^2},\]
for $a= \sqrt{\kappa n\Delta^{-1}\pi_U\eps_n^2 Dn\Delta\eps_n^2}=n\eps_n^2\sqrt{\kappa\pi_U D}$ and $b=\kappa \Delta^{-1} D n\Delta\eps_n^2=n\eps_n^2\kappa D$, for some constant $\kappa=\kappa(\Ii)$. We see that $a/n$ and $b/n$ are both upper bounded by a constant multiple of $\eps_n^2$, hence, by choosing $C$ large enough, \Cref{eqn:normnConcentrates'} holds. 

\paragraph{Bounding $P_b\brackets[\big]{\norm{\tilde{b}_n-b}_n^2\II_{\Aa_n\cap \Omega_n}>C\eps_n^2}$:}
We show that $P_b\brackets[\big]{\norm{\tilde{b}_n-b}_n^2\II_{\Aa_n\cap \Omega_n}>C\eps_n^2}\leq 8e^{-Dn\Delta\eps_n^2}$ for some constant $C$.

Recall an application of \Cref{eqn:WaveletSeriesConverges} showed us that $\norm{\pi_{l_n}b}_\infty \leq K_0+1$ for sufficiently large $n$, hence we see that $\pi_{l_n}b$ lies in $\tilde{S}_{l_n}$, so by definition $\gamma_n(\tilde{b}_n)\leq \gamma_n(\pi_{l_n}b)$. We now use this to show that \begin{equation}\label{eqn:empiricalNormControl} \frac{1}{4}\norm{\tilde{b}_n-b}_n^2 \II_{\Aa_n\cap \Omega_n}\leq \frac{7}{4}\norm{\pi_{l_n}b-b}_n^2+8\nu_n(\tilde{t}_n)^2\II_{\Aa_n}+ \frac{8}{n}\sum_{k=1}^n R_{k\Delta}^2,\end{equation} 
where $\nu_n(t)=\frac{1}{n}\sum_{k=1}^n t(X_{k\Delta})Z_{k\Delta}$ and we recall that $\tilde{t}_n=(\tilde{b}_n-\pi_{l_n}b)\norm{\tilde{b}_n-\pi_{l_n}b}_\mu^{-1}$. The argument, copied from \cite{Comte2007} Sections 3.2 and 6.1, is as follows.
Using $\Delta^{-1}(X_{(k+1)\Delta}-X_{k\Delta})=b(X_{k\Delta})+Z_{k\Delta}+R_{k\Delta}$ and $\gamma_n(\tilde{b}_n)-\gamma_n(b)\leq \gamma_n(\pi_{l_n}b)-\gamma_n(b)$, one shows that 
\begin{equation}\label{eqn:empiricalNormControlIntermediateStep}\norm{\tilde{b}_n-b}_n^2 \leq \norm{\pi_{l_n}b-b}_n^2 + 2\nu(\tilde{b}_n-\pi_{l_n}b)+\frac{2}{n}\sum_{k=1}^n R_{k\Delta} (\tilde{b}_n-\pi_{l_n}b)(X_{k\Delta}).\end{equation}
Repeatedly applying the AM-GM--derived inequality $2ab\leq 8 a^2+b^2/8$ yields
\begin{align*} \frac{2}{n}\sum_{k=1}^n R_{k\Delta} (\tilde{b}_n-\pi_{l_n}b)(X_{k\Delta})&\leq \frac{8}{n}\sum_{k=1}^n R_{k\Delta}^2 + \frac{1}{8} \norm{\tilde{b}_n-\pi_{l_n}b}_n^2,\\
2\nu(\tilde{b}_n-\pi_{l_n}b)=2\norm{\tilde{b}_n-\pi_{l_n}b}_\mu \nu(\tilde{t}_n)&\leq 8\nu_n(\tilde{t}_n)^2+\frac{1}{8}\norm{\tilde{b}_n-\pi_{l_n}b}_\mu^2.\end{align*} 

Next recall that on $\Aa_n\cap \Omega_n$, we have $\norm{\tilde{b}_n-\pi_{l_n}b}_\mu^2\leq 2\norm{\tilde{b}_n-\pi_{l_n}b}_n^2,$ and further recall $\norm{\tilde{b}_n-\pi_{l_n}b}_n^2\leq 2\norm{\tilde{b}_n-b}_n^2+2\norm{\pi_{l_n}b-b}_n^2$. Putting all these bounds into \Cref{eqn:empiricalNormControlIntermediateStep} yields \Cref{eqn:empiricalNormControl}, where on the right hand side we have only included indicator functions where they will help us in future steps. 
Next, by a union bound, we deduce 
\begin{multline*}P_b(\norm{\tilde{b}_n-b}_n^2\II_{\Aa_n\cap \Omega_n}>C\eps_n^2) \\ \leq  P_b(\norm{\pi_{l_n}b-b}_n^2 >C'\eps_n^2)+P_b(\nu_n(\tilde{t}_n)^2\II_{\Aa_n} >C'\eps_n^2)+P_b\brackets[\Big]{\frac{1}{n}\sum_{k=1}^n R_{k\Delta}^2 > C'\eps_n^2},\end{multline*}
for some constant $C'$ (we can take $C'=C/96$). We have already shown that $P_b(\norm{\pi_{l_n}b-b}_n>C\eps_n)\leq 2e^{-Dn\Delta\eps_n^2}$ for a large enough constant $C$, thus the following two lemmas conclude the proof.
\end{proof}

\begin{lemma}\label{lem:Rkdelta2} Under the conditions of \Cref{thm:bhat_concentrates}, for each $D>0$ there exists a constant $C=C(\Ii,L_0,D)>0$ for which, for $n$ sufficiently large,
$P_{b}\brackets*{\frac{1}{n}\sum_{k=1}^n R_{k\Delta}^2 > C \eps_n^2}\leq 2e^{-Dn\Delta\eps_n^2}.$ 
\end{lemma}
\begin{lemma}\label{lem:nu}
Under the conditions of \Cref{thm:bhat_concentrates}, for each $D>0$ there exists a constant $C=C(\Ii,L,D)>0$ for which, for $n$ sufficiently large, $P_{b}(\nu_n(\tilde{t}_n)\II_{\Aa_n} > C\eps_n)\leq 4e^{-Dn\Delta\eps_n^2}.$
\end{lemma}

\begin{proof}[Proof of \Cref{lem:Rkdelta2}]

Recall $R_{k\Delta}=\frac{1}{\Delta}\int_{k\Delta}^{(k+1)\Delta} (b(X_s)-b(X_{k\Delta}))\ds,$ and recall any $b\in\Theta$ is Lipschitz, with Lipschitz constant at most $K_0$, so $\abs{R_{k\Delta}}\leq K_0 \max_{s\leq \Delta} \abs{X_{k\Delta+s}-X_{k\Delta}}.$ 
It is therefore enough to bound $\sup\braces{\abs{X_t-X_s} : \: s,t\in [0,n\Delta],~ \abs{t-s}\leq \Delta }$. 

We apply the H{\"o}lder continuity result (\Cref{lem:X_holder}) with $u=D^{1/2}\lambda^{-1/2}(n\Delta\eps_n^2)^{1/2}$ for $\lambda=\lambda(\Ii)$ the constant of the \namecref{lem:X_holder}, noting that the assumption $n\Delta\eps_n^2/\log(n\Delta)\to \infty$ ensures that $u$ is large enough compared to $m=n\Delta$ that the conditions for the \namecref{lem:X_holder} are met, at least when $n$ is large. We see that 
\[\sup_{\substack{s,t\in[0,n\Delta] \\ \abs{t-s}\leq \Delta}} \abs{X_{t}-X_{s}}\leq \Delta^{1/2}\brackets*{\log(n\Delta) ^{1/2}+\log(\Delta^{-1})^{1/2}} D^{1/2}\lambda^{-1/2}(n\Delta\eps_n^2)^{1/2},\] on an event $\Dd$ of probability at least $1-2e^{-Dn\Delta\eps_n^2}$, (we have used that, for $n$ large enough, $\Delta\leq \min(\tau,e^{-1})$ in order to take the supremum over $\abs{t-s}\leq \Delta$ and to see $\sup_{\delta\leq \Delta}w_m(\delta)=w_m(\Delta)$).

Now observe that $\log(n\Delta)^{1/2}\leq(\log(\Delta^{-1})^{1/2})$ for large enough $n$ because $n\Delta^2\to 0$ (so $n\Delta\leq \Delta^{-1}$ eventually). Further, from the assumption $n\Delta^2 \log(\Delta^{-1})\leq L_0$ we are able to deduce that  $\Delta^{1/2}\log(\Delta^{-1})^{1/2}(n\Delta\eps_n^2)^{1/2}\leq L_0^{1/2}\eps_n$. It follows that on $\Dd$, we have $R_{k\Delta}\leq C\eps_n$ for a suitably chosen constant $C$ (independent of $k$ and $n$), which implies the desired concentration.
\end{proof}

\begin{proof}[Proof of \Cref{lem:nu}] Recall for $Z_{k\Delta}=\frac{1}{\Delta}\int_{k\Delta}^{(k+1)\Delta} \sigma(X_s)\dW_s$ we set $\nu_n(t)=\frac{1}{n}\sum_{k=1}^n t(X_{k\Delta})Z_{k\Delta}.$  
The martingale-derived concentration result Lemma 2 in Comte et al.\ \cite{Comte2007} (the model assumptions in \cite{Comte2007} are slightly different to those made here, but the proof of the lemma equally applies in our setting) tells us 
$ P_{b}(\nu_n(t) \geq \xi,\norm{t}_n^2\leq u^2) \leq \exp\brackets*{-\frac{n\Delta\xi^2}{2\sigma_U^2 u^2}},$ for any $t,u$, and for any drift function $b\in\Theta$, so that
\begin{equation*} \tag{$\star$}\label{eqn:nuConcentrates} P_{b}(\nu_n(t) \geq \xi) \leq \exp\brackets*{-\frac{n\Delta\xi^2}{2\sigma_U^2 u^2}}+P_{b}\brackets{\norm{t}_n^2>u^2}.\end{equation*}
We can apply \Cref{thm:concentration-for-diffusions} to see that, for some constant $\kappa=\kappa(\Ii)$, \begin{align*}
P_{b}\brackets{\norm{t}_n^2>u^2} &= P_{b}\brackets*{\frac{1}{n}\brackets*{\sum_{k=1}^n t(X_{k\Delta})^2-\norm{t}_\mu^2} >u^2-\norm{t}_\mu^2}\\
&\leq \exp\brackets*{-\frac{1}{\kappa} \Delta \min\braces*{ \frac{n^2(u^2-\norm{t}_\mu^2)^2}{n\norm{t^2}_\mu^2},\frac{n(u^2-\norm{t}_\mu^2)}{\norm{t^2}_\infty}}}\\
&\leq \exp\brackets*{-\frac{1}{\kappa} n\Delta (u^2-\norm{t}_\mu^2)\norm{t}_\infty^{-2}\min \brackets{u^2\norm{t}_\mu^{-2} -1,1}},
\end{align*}
where to obtain the last line we have used that $\norm{t^2}_\mu^2 \leq \norm{t}_\infty^2 \norm{t}_\mu^2$.

Now choose $u^2=\norm{t}_\mu^2+\xi\norm{t}_\infty$. Then $\xi^2/u^2 \geq \frac{1}{2}\min(\xi^2/\norm{t}_\mu^2,\xi/\norm{t}_\infty)$ so that, returning to \Cref{eqn:nuConcentrates}, we find
\begin{align*}
P_{b}(\nu_n(t) \geq \xi) &\leq \exp\brackets*{-\frac{n\Delta}{4\sigma_U^2} \min\brackets{\xi^2\norm{t}_\mu^{-2},\xi\norm{t}_\infty^{-1}}}  + \exp\brackets[\Big]{-\frac{1}{\kappa} n\Delta \xi  \min(\xi\norm{t}_\mu^{-2},\norm{t}_\infty^{-1})}\\
& \leq 2 \exp\brackets*{-\frac{1}{\kappa'} n\Delta \min\brackets{\xi^2\norm{t}_\mu^{-2},\xi\norm{t}_\infty^{-1}}},
\end{align*} for some constant $\kappa'=\kappa'(\Ii)$. 

By changing variables we attain the bound
$P_{b}\brackets{\nu_n(t) \geq \max\brackets{\sqrt{v^2x},ux}} \leq 2 \exp\brackets*{-x},$
where $v^2=\kappa' (n\Delta)^{-1}\norm{t}_\mu^2$ and $u=\kappa' (n\Delta)^{-1} \norm{t}_\infty$. 
Then, as in \Cref{thm:concentration-for-diffusions}, a standard chaining argument allows us to deduce that
\[P_{b} \brackets*{\sup_{t\in I_n}\nu_n(t)\geq \tilde{\kappa} \brackets[\Big]{\sqrt{V^2(D_{l_n}+x)}+U(D_{l_n}+x)}}\leq 4 e^{-x},\]
for $V^2=\sup_{t\in I_n} \norm{t}_\mu^2 (n\Delta)^{-1}=(n\Delta)^{-1}$, $U=\sup_{t\in I_n} \norm{t}_\infty (n\Delta)^{-1}=C_1\eps_n^{-1}(n\Delta)^{-1}$, and for a constant $\tilde{\kappa}=\tilde{\kappa}(\Ii)$. Taking $x=Dn\Delta\eps_n^2$ and recalling the assumption $D_{l_n}\leq L n\Delta \eps_n^2$ we obtain the desired result (conditional on $\tilde{t}_n\in I_n$, which is the case on the event $\Aa_n$).
\end{proof}

\section{Small ball probabilities}\label{sec:SmallBallProbs}
Now we show that the Kullback--Leibler divergence between the laws corresponding to different parameters $b_0,b$ can be controlled in terms of the $L^2$--distance between the parameters.
Denote by $K(p,q)$ the Kullback--Leibler divergence between probability distributions with densities $p$ and $q$, i.e.\ $K(p,q)= E_p\log(\frac{p}{q})=\int\log(\frac{p(x)}{q(x)}){\dif p}(x).$ Also write
\[\KL(b_0,b)=~~E_{b_0} \sqbrackets*{ \log\brackets*{\frac{p_0(\Delta,X_0,X_\Delta)}{p_b(\Delta,X_0,X_\Delta)}}}. \] 
Recalling that $p_b^{(n)}(x^{(n)})=\pi_b(x_0)\prod_{i=1}^{n}p_b(\Delta,x_{(i-1)\Delta},x_{i\Delta})$ is the density on $\RR^{n+1}$ of $X^{(n)}$ under $P_b$, we introduce the following Kullback--Leibler type neighbourhoods: for $\eps>0$, define
\begin{align*}\label{eqn:BKLdefinition} &B_{KL}^{(n)}(\eps)=\braces*{b\in \Theta \st K(p_0^{(n)},p_b^{(n)})\leq (n\Delta+1)\eps^2,~ \Var_{b_0}\brackets[\Big]{\log\frac{p_0^{(n)}}{p_b^{(n)}}}\leq (n\Delta+1)\eps^2}, \\
&  B_\eps=\braces*{b\in \Theta \st K(\pi_0,\pi_b)\leq \eps^2,~\Var_{b_0}\brackets[\Big]{\log\frac{\pi_0}{\pi_b}}\leq \eps^2,~\KL(b_0,b)\leq \Delta \eps^2,~\Var_{b_0}\brackets[\Big]{\log\frac{p_0}{p_b}}\leq \Delta\eps^2}.
\end{align*}
Note that $\KL(b_0,b)$ and $B_\eps$ implicitly depend on $n$ via $\Delta$. 

The main result of this section is the following. 

\begin{theorem}\label{thm:SmallBallProbs} 
Consider data $X^{(n)}=(X_{k\Delta})_{0\leq k \leq n}$ sampled from a solution $X$ to \Cref{eqn:sde} under \Cref{assumptions:b,assumptions:sigma,assumption:Delta,assumption:invariant}. Let $\eps_n\to0$ be a sequence of positive numbers such that $n\Delta\eps_n^2\to \infty$. Then there is a constant $A=A(\Ii)$ such that, for all $n$ sufficiently large, $\braces{b\in \Theta \st \norm{b-b_0}_2\leq A\eps_n}\subseteq B_{KL}^{(n)}(\eps_n)$. 
\end{theorem}
\begin{proof}
Applying \Cref{lem:VarianceTensorises} in the appendix where it is shown that \[\Var_{b_0} \log\brackets*{\frac{p_0^{(n)}(X^{(n)})}{p_b^{(n)}(X^{(n)})}}
\leq 3\Var_{b_0}\brackets*{\log\frac{\pi_0(X_0)}{\pi_b(X_0)}} +3n\Var_{b_0}\brackets*{\log\frac{p_0(X_0,X_\Delta)}{p_b(X_0,X_\Delta)}}, \] and noting also that $K(p_0^{(n)},p_b^{(n)})=K(\pi_0,\pi_b)+n\KL(b_0,b)$ by linearity, we observe that $B_{\eps_n/\sqrt{3}}\subseteq B_{KL}^{(n)}(\eps_n)$. It is therefore enough to show that for some $A=A(\Ii)$ we have $\braces{b\in \Theta \st \norm{b-b_0}_2\leq A\eps_n}\subseteq B_{\eps_n/\sqrt{3}}$. This follows immediately by applying \Cref{lem:SmallBallProbs} below to $\xi_n=\eps_n/\sqrt{3}$.
\end{proof}

\begin{lemma} \label{lem:SmallBallProbs}
Under the conditions of \Cref{thm:SmallBallProbs}, there is an $A=A(\Ii)$ such that, for all $n$ sufficiently large, $\braces{b\in \Theta \st \norm{b-b_0}_2\leq A\eps_n}\subseteq B_{\eps_n}$.
\end{lemma} 

The key idea in proving \Cref{lem:SmallBallProbs} is to use the Kullback--Leibler divergence between the laws $P_{b_0}^{(x)},P_b^{(x)}$ of the continuous-time paths to control the Kullback--Leibler divergence between $p_b$ and $p_0$. This will help us because we can calculate the Kullback--Leibler divergence between the full paths using Girsanov's Theorem, which gives us an explicit formula for the likelihood ratios.

Let $P_{b,T}^{(x)}$ denote the law of $(X_t)_{0\leq t \leq T}$ conditional on $X_0=x$, i.e.\ the restriction of $P_b^{(x)}$ to $C([0,T])$. We write $\WW_{\sigma,T}^{(x)}$ for $P_{b,T}^{(x)}$ when $b=0$. Throughout this section we will simply write $P_b^{(x)}$ for $P_{b,\Delta}^{(x)}$ and similarly with $\WW_\sigma^{(x)}$. We have the following.

\begin{theorem}[Girsanov's Theorem]\label{thm:Girsanov}
Assume $b_0$ and $b$ lie in $\Theta$, and $\sigma$ satisfies \Cref{assumptions:sigma}. 
Then the laws $P_{b_0,T}^{(x)}$ and $P_{b,T}^{(x)}$ are mutually absolutely continuous with, for $X\sim P_{b,T}^{(x)}$, the almost sure identification 
\[\od{P_{b_0,T}^{(x)}}{P_{b,T}^{(x)}}((X_t)_{t\leq T})=\exp\sqbrackets*{\int_0^T \frac{b_0-b}{\sigma^2}(X_t)\dX_t - \frac{1}{2} \int_0^T \frac{b_0^2-b^2}{\sigma^2}(X_t)\dt}.\]
\end{theorem}
\begin{proof}
See Liptser \& Shiryaev \cite{Liptser1977}, Theorem 7.19, noting that the assumptions are met because $b,b_0$ and $\sigma$ are all Lipschitz and bounded, and $\sigma$ is bounded away from 0.  
\end{proof}

We write \begin{equation} \label{eqn:tildep} \tilde{p}_0^{(x)}=\od{P_{b_0}^{(x)}}{\WW^{(x)}_\sigma}, \qquad \tilde{p}_b^{(x)}=\od{P_b^{(x)}}{\WW_\sigma^{(x)}}\end{equation} for the Radon-Nikodym derivatives (i.e.\ densities on $C([0,\Delta])$ with respect to $\WW_\sigma^{(x)}$) whose existence Girsanov's Theorem guarantees. 
We will simply write $X$ for $(X_t)_{t\leq \Delta}$ where context allows, and similarly with $U$. Since $\tilde{p}_0^{(x)}(X)=0$ for any path $X$ with $X_0\not = x$, we will further omit the superscripts on our densities in general, writing $\tilde{p}_0(X)$ for $\tilde{p}_0^{(X_0)}(X)$, and similarly for $\tilde{p}_b$.

\begin{proof}[Proof of \Cref{lem:SmallBallProbs}]
We break the proof into a series of lemmas. We will upper bound the variances in the definition of $B_{\eps_n}$ by the corresponding uncentred second moments. For some constant $A=A(\Ii)$ we show the following.
\begin{enumerate}[1.]
\item $A^2\KL(b_0,b)\leq \Delta\norm{b-b_0}_2^2,$ which shows that $\KL(b_0,b)\leq \Delta \eps_n^2$ whenever $\norm{b-b_0}_2\leq A\eps_n$. This is the content of \Cref{lem:KL-bounded}.
\item If $\norm{b-b_0}_2\leq A\eps_n$ then we have $E_{b_0} [\log(p_0/p_b)^2]\leq \Delta \eps_n^2.$ This is the content of \Cref{lem:KL2-bounded}. Note that the other steps do not need any assumptions on $\eps_n$, but this step uses $n\Delta\eps_n^2\to \infty$.
\item $A^2\max\braces*{K(\pi_0,\pi_b),E_{b_0}[\log(\pi_0/\pi_b)^2]}\leq \norm{b_0-b}_2^2.$ From this it follows that $K(\pi_0,\pi_b)\leq \eps_n^2$ and $E_{b_0}[\log(\pi_0/\pi_b)^2]\leq \eps_n^2$ whenever $\norm{b-b_0}_2\leq A\eps_n$. This is the content of \Cref{lem:Kmu-bounded}. 
\end{enumerate}
Together, then, the three lemmas below conclude the proof. \end{proof}

\begin{lemma}\label{lem:KL-bounded} Under the conditions of \Cref{thm:SmallBallProbs}, there is a constant $A$ depending only on $\Ii$ such that $A^2\KL(b_0,b)\leq \Delta \norm{b_0-b}_2^2$.

\end{lemma}
The proof is essentially the same as that in van der Meulen \& van Zanten \cite{vanderMeulenvanZanten2013} Lemma 5.1, with minor adjustments to fit the periodic model and non-constant $\sigma$ used here. Further, all the ideas needed are exhibited in the proof of \Cref{lem:KL2-bounded}. Thus, we omit the proof.

\begin{lemma}\label{lem:KL2-bounded} Under the conditions of \Cref{thm:SmallBallProbs}, there is a constant $A=A(\Ii)$ so that, for $n$ sufficiently large, $E_{b_0}[\log(p_0/p)^2]\leq \Delta \eps_n^2$ whenever $\norm{b-b_0}_2\leq A\eps_n$.
\end{lemma}
\begin{proof}
We first show that we can control the second moment of $\log(p_0/p_b)$ by the second moment of the corresponding expression $\log(\tilde{p}_0/\tilde{p}_b)$ for the full paths, up to an approximation error which is small when $\Delta$ is small.
Consider the smallest convex function dominating $\log(x)^2$, given by 
\[ h(x)=\begin{cases} \log(x)^2 & x<e\\ 2e^{-1}x-1 & x\geq e
\end{cases}\]
(it is in fact more convenient, and equivalent, to think of $h$ as dominating the function $x\mapsto (\log x^{-1})^2$).
Let $X\sim P_{b_0}^{(x)}$ and let $U\sim \WW_{\sigma}^{(x)}$. Intuitively, the probability density of a transition of $X$ from $x$ to $y$, with respect to the (Lebesgue) density $p_*$ of transitions of $U$ from $x$ to $y$, can be calculated by integrating the likelihood $\tilde{p}_0(U)$ over all paths of $U$ which start at $x$ and end at $y$, and performing this integration will yield the conditional expectation of $\tilde{p}_0^{(x)}(U)$ given $U_\Delta$. That is to say,
\begin{equation}\label{eqn:transitionDensities} \frac{p_0(\Delta,x,y)}{p_*(\Delta,x,y)}=E_{\WW_\sigma^{(x)}}\sqbrackets*{\tilde{p}_0(U) \mid U_\Delta =y}.\end{equation}
The above argument is not rigorous because we condition on an event of probability zero, but the formula \Cref{eqn:transitionDensities} is true, and is carefully justified in \Cref{lem:DiffusionBridge} in \Cref{sec:TechnicalLemmas}.
A corresponding expression holds for $p_b(\Delta,x,y)$, so that
\[ 
E_{b_0} \sqbrackets*{\log\brackets[\Big]{\frac{p_0(\Delta,X_0,X_\Delta)}{p_b(\Delta,X_0,X_\Delta)}}^2}\leq E_{b_0}[h(p_b/p_0)]=E_{b_0}\sqbrackets*{h \brackets[\bigg]{\frac{E_{\WW_\sigma^{(X_0)}}[\tilde{p}_b(U) \mid U_\Delta=X_\Delta]}{E_{\WW_\sigma^{(X_0)}
}[\tilde{p}_0(U) \mid U_\Delta=X_\Delta]}}}.
\]
\Cref{lem:abstractBayes} in \Cref{sec:TechnicalLemmas} allows us to simplify the ratio of conditional expectations. We apply with $\PP=\WW_\sigma^{(X_0)}$, $\QQ=P_{b_0}^{(X_0)}$ and $g=\tilde{p}_b^{(X_0)}/\tilde{p}_0^{(X_0)}$, then further apply conditional Jensen's inequality and the tower law to find
\begin{align*}
E_{b_0} \sqbrackets*{\brackets[\Big]{\log\frac{p_0}{p_b}}^2} 
&\leq  E_{b_0} \sqbrackets*{h\brackets[\Big]{E_{P_{b_0}^{(X_0)}}\sqbrackets[\Big]{\frac{\tilde{p}_b} {\tilde{p}_0 }( X
) \mid X_\Delta}}}\leq E_{b_0} \sqbrackets[\Big]{h\brackets[\Big]{\frac{\tilde{p}_b}{\tilde{p}_0}(X)}} \\
& \qquad \leq E_{b_0}\sqbrackets*{\brackets*{\log\frac{\tilde{p}_0}{\tilde{p}_b}(X)}^2 }+E_{b_0}\sqbrackets*{(2e^{-1}\frac{\tilde{p}_b}{\tilde{p}_0}(X)-1)\II\braces[\Big]{\frac{\tilde{p}_b}{\tilde{p}_0}(X)\geq e}},\end{align*} 
which is the promised decomposition into a corresponding quantity for the continuous case and an approximation error.
We conclude by showing that each of these two terms is bounded by $\frac{1}{2}\Delta\eps_n^2$, provided $\norm{b-b_0}_2\leq A\eps_n$ for some sufficiently small constant $A=A(\Ii)$.
\begin{paragraph}{Showing $E_{b_0}\sqbrackets*{\brackets*{\log\frac{\tilde{p}_0}{\tilde{p}_b}}^2}\leq \frac{1}{2}\Delta\eps_n^2$:}

Write $f=\frac{b_0-b}{\sigma}$. Then we apply Girsanov's Theorem (\Cref{thm:Girsanov}) to find
\begin{align*} E_{b_0}\sqbrackets*{\brackets*{\log\frac{\tilde{p}_0}{\tilde{p}_b}(X)}^2} 
&= E_{b_0} \sqbrackets*{\brackets[\Big]{ 
\int_0^\Delta f(X_t)\dW_t+\frac{1}{2}\int_0^\Delta  f^2(X_t) \dt}^2 }, \\
&=E_{b_0}\sqbrackets[\Big]{\brackets[\Big]{\int_0^\Delta f(X_t)\dW_t}^2} +\frac{1}{4}E_{b_0}\sqbrackets[\Big]{\brackets[\big]{\int_0^\Delta  f^2(X_t) \dt}^2 } \end{align*}
The cross term has vanished in the final expression because $\int_0^\Delta f(X_t)\dW_t$ is a martingale for $X\sim P_{b_0}$ (since $f$ is bounded thanks to \Cref{assumptions:b,assumptions:sigma} and a bounded semimartingale integrated against a square integrable martingale yields a martingale, as in \cite{Rogers2000} IV.27.4), while $\int_0^\Delta f^2(X_t)\dt$ is a finite variation process, and the expectation of a martingale against a finite variation process is zero (eg.\ see \cite{Rogers2000} IV.32.12). 

For the first term on the right, we use It{\^o}'s isometry (\cite{Rogers2000} IV.27.5), Fubini's Theorem, periodicity of $f$ and stationarity of $\mu_0$ for the periodised process $\dot{X}=X\mod 1$ to find 
\[E_{b_0}\brackets[\Big]{\int_0^\Delta f(X_t)\dW_t}^2=E_{b_0}\int_0^\Delta f^2(X_t)\dt=\int_0^\Delta E_{b_0} f^2(\dot{X}_t) \dt = \Delta \norm{f}_{\mu_0}^2.\]
 
The second term $\frac{1}{4}E_{b_0}\sqbrackets[\Big]{\brackets[\big]{\int_0^\Delta  f^2(X_t) \dt}^2}$  is upper bounded by $\frac{1}{4} \Delta^2 \norm{f}_\infty^2\norm{f}_{\mu_0}^2$ (this can be seen from the bound $(\int_0^\Delta f^2)^2\leq \Delta\norm{f}_\infty^2 \int_0^\Delta f^2$), hence is dominated by $\Delta\norm{f}_{\mu_0}^2$ when $n$ is large. Thus, for some constant $A=A(\Ii)$ we find
\[E_{b_0}\sqbrackets*{\brackets*{\log\frac{\tilde{p}_0}{\tilde{p}_b}(X)}^2}\leq  2 \Delta\norm{f}_{\mu_0}^2\leq \frac{1}{2} A^{-2} \Delta\norm{b_0-b}_2^2,\] where \Cref{assumptions:b,assumptions:sigma} allow us to upper bound $\norm{f}_{\mu_0}$ by $\norm{b_0-b}_2$, up to a constant depending only on $\Ii$. For $\norm{b_0-b}_2\leq A\eps_n$ we then have $E_{b_0}\sqbrackets[\big]{\brackets[\big]{\log(\tilde{p}_b/\tilde{p}_0)}^2 }\leq \Delta\eps_n^2/2.$
	
\end{paragraph}
\begin{paragraph}{Showing $E_{b_0} \sqbrackets*{(2e^{-1}\frac{\tilde{p}_b}{\tilde{p}_0}(X)-1)\II\braces{\frac{\tilde{p}_b}{\tilde{p}_0}(X)\geq e}}\leq \frac{1}{2}\Delta\eps_n^2$:}

We have 
\[E_{b_0} \sqbrackets[\Big]{\brackets[\Big]{2e^{-1}\frac{\tilde{p}_b}{\tilde{p}_0}(X)-1}\II\braces[\Big]{\frac{\tilde{p}_b}{\tilde{p}_0}(X)\geq e}} \leq 2e^{-1}P_b\sqbrackets[\Big]{\frac{\tilde{p}_b}{\tilde{p}_0} \geq e} \leq P_b\sqbrackets[\Big]{\log\brackets[\Big]{\frac{\tilde{p}_b}{\tilde{p}_0}(X)}\geq 1}.\]
By the tower law it suffices to show $P_b^{(x)}\sqbrackets[\Big]{\log\brackets[\Big]{\frac{\tilde{p}_b}{\tilde{p}_0}(X)}\geq 1}\leq \frac{1}{2}\Delta\eps_n^2$ for each $x\in[0,1]$.
Applying Girsanov's Theorem (\Cref{thm:Girsanov}) we have, for $f=(b_0-b)/\sigma$, and for $n$ large enough that $\Delta\norm{f}_\infty^2\leq 1$,
\begin{align*}
P_{b}^{(x)}\brackets[\Big]{\log\frac{\tilde{p}_b}{\tilde{p}_0}(X)>1}
&=P_{b}^{(x)}\brackets[\Big]{\int_0^\Delta -f(X_t)\dW_t+\frac{1}{2}\int_0^\Delta f(X_t)^2\dt >1}\\
&\leq P_{b}^{(x)}\brackets[\Big]{\int_0^\Delta -f(X_t)\dW_t> 1/2 }.
\end{align*}
Write $M_t=\int_0^t -f(X_s) \dW_s$. Then, for $A=\max(1,(2K_0/\sigma_L)^2)$, since $A$ uniformly upper bounds $\norm{f}_\infty^2$ for $b\in \Theta$, we see that $M$ is a martingale whose quadratic variation satisfies $\abs{\qv{M}_t-\qv{M}_s}\leq A\abs{t-s}$. Recalling that $w_1(\delta)=\delta^{1/2}\log(\delta^{-1})^{1/2}$, we apply \Cref{lem:Y_holder} with $u=w_1(\Delta)^{-1}/2$ to yield that, for $n$ large enough, 
\begin{align*}P_{b}^{(x)}\brackets[\Big]{\log\frac{\tilde{p}_b}{\tilde{p}_0}(X)>1}&\leq P_{b}^{(x)}\brackets[\Big]{\sup_{s,t\leq \Delta, s\not = t} \frac{\abs{M_t-M_s}}{w_1(\abs{t-s})} >\frac{1}{2}w_1(\Delta)^{-1} }\\
&\leq 2\exp\brackets[\Big]{-\lambda  w_1(\Delta)^{-2}},
 \end{align*}
where $\lambda$ is a constant depending only on $\Ii$.

Recall we assume $n\Delta\to \infty$ and $n\Delta^2 \to 0$. It follows that for large enough $n$ we have $\log(\Delta^{-1})\leq \log(n)$, and $\Delta\leq \lambda \log(n)^{-2}$.
Then observe
\begin{align*}
&\Delta\leq \lambda \log(n)^{-2} \implies \Delta \leq \lambda (\log \Delta^{-1})^{-1}\log(n)^{-1} \implies \log(n)\leq\lambda \Delta^{-1}(\log \Delta^{-1})^{-1},
\end{align*}
so that $\exp\brackets[\big]{-\lambda  w_1(\Delta)^{-2}}\leq n^{-1}$ for $n$ large. 
Finally, since $n\Delta\eps_n^2\to \infty$, we see $2n^{-1}\leq \frac{1}{2}\Delta\eps_n^2$ for $n$ large enough, as required.
\end{paragraph}
\end{proof}

\begin{lemma}\label{lem:Kmu-bounded}  Under the conditions of \Cref{thm:SmallBallProbs}, there is a constant $A$ depending only on $\Ii$ such that $A^2\max\braces*{K(\pi_0,\pi_b),E_{b_0}[\log(\pi_0/\pi_b)^2]}\leq \norm{b_0-b}_2^2.$\end{lemma}
\begin{proof}

By the comment after Lemma 8.3 in \cite{GGV2000}, it suffices to prove that $h^2(\pi_0,\pi_b)\norm{\pi_0/\pi_b}_\infty\leq C\norm{b-b_0}_2^2$ for some $C=C(\Ii)$, where $h$ is the Hellinger distance between densities defined by $h^2(p,q)=\int (\sqrt{p}-\sqrt{q})^2$. Since $\pi_0,\pi_b$ are uniformly bounded above and away from zero, we can absorb the term $\norm{\pi_0/\pi_b}_\infty$ into the constant. 

We initially prove pointwise bounds on the difference between the densities $\pi_0,\pi_b$.
Recall we saw in \Cref{sec:FrameworkAndAssumptions} that, for $I_b(x)=\int_0^x \frac{2b}{\sigma^2}(y)\dy$, we have
\begin{align*} &\pi_b(x) =\frac{e^{I_b(x)}}{H_b\sigma^2(x)}\brackets[\Big]{e^{I_b(1)}\int_x^1 e^{-I_b(y)}\dy +\int_0^x e^{-I_b(y)}\dy}, \qquad x\in [0,1], \\ & H_b=\int_0^1 \frac{e^{I_b(x)}}{\sigma^2(x)}\brackets[\Big]{e^{I_b(1)}\int_x^1 e^{-I_b(y)}\dy +\int_0^x e^{-I_b(y)}\dy}\dx. \end{align*}

We can decompose:
$ \abs{\pi_b(x)-\pi_0(x)}\leq D_1+D_2+D_3+D_4,$
where \begin{align*}
&D_1=\frac{e^{I_b(x)}}{\sigma^2(x)} \abs[\Big]{\frac{1}{H_b}-\frac{1}{H_{b_0}}}  \brackets[\Big]{e^{I_b(1)}\int_x^1 e^{-I_b(y)}\dy+\int_0^x e^{-I_b(y)}\dy}, \\
&D_2= \frac{\abs{e^{I_b(x)}-e^{I_{b_0}(x)}}}{H_{b_0}\sigma^2(x)} \brackets[\Big]{e^{I_b(1)}\int_x^1 e^{-I_b(y)}\dy+\int_0^x e^{-I_b(y)}\dy}, \\
&D_3=\frac{e^{I_{b_0}(x)}}{H_{b_0}\sigma^2(x)}\abs[\Big]{\brackets[\big]{e^{I_b(1)}-e^{I_{b_0}(1)}}\int_x^1 e^{-I_b(y)}\dy },\\
&D_4=\frac{e^{I_{b_0}(x)}}{H_{b_0}\sigma^2(x)}\abs[\Bigg]{e^{I_{b_0}(1)}\int_x^1 \brackets{e^{-I_b(y)}-e^{-I_{b_0}(y)}}\dy +\int_0^x \brackets{e^{-I_b(y)}-e^{-I_{b_0}(y)}}\dy}.
\end{align*}
We have the bounds $
 \sigma_U^{-2}e^{-6K_0\sigma_L^{-2}}\leq H_b\leq \sigma_L^{-2}e^{6K_0\sigma_L^{-2}},$ and $ e^{-2K_0\sigma_L^{-2}}\leq e^{I_b(x)} \leq e^{2K_0\sigma_L^{-2}}. $ 
An application of the mean value theorem then tells us \[
\abs[\Big]{e^{I_b(x)}-e^{I_{b_0}(x)}}\leq C(\Ii) \int_0^x \frac{2\abs{b_0-b}}{\sigma^2}(y)\dy
\leq C'(\Ii)\norm{b_0-b}_2,
\] for some constants $C$, $C'\!$, and the same expression upper bounds $\abs{e^{-I_b(x)}-e^{-I_{b_0}(x)}}$.

It follows that, for some constant $C=C(\Ii)$, we have $D_i\leq C \norm{b-b_0}_2$ for $i=2,3,4$. For $i=1$ the same bound holds since $\abs{\frac{1}{H_b}-\frac{1}{H_{b_0}}}\leq \frac{\abs{H_b-H_{b_0}}}{H_b H_{b_0}}$ and a similar decomposition to the above yields $\abs{H_b-H_{b_0}}\leq C(\Ii)\norm{b-b_0}_2$. 

Thus, we have shown that $\abs{\pi_b(x)-\pi_0(x)}\leq C(\Ii)\norm{b-b_0}_2$. Integrating this pointwise bound, we find that $\norm{\pi_0-\pi_b}_2\leq C(\Ii)\norm{b_0-b}_2$. 
Finally, since $h^2(\pi_0,\pi_b)\leq \frac{1}{4\pi_L}\norm{\pi_0-\pi_b}_2^2\leq C'(\Ii)\norm{b_0-b}_2^2,$ for some different constant $C'$, we are done. 
\end{proof}

\section{Main contraction results: proofs} \label{sec:MainContractionResultsProof}
We now have the tools we need to apply general theory in order to derive contraction rates.
Recall that $K(p,q)$ denotes the Kullback--Leibler divergence between probability distributions with densities $p$ and $q$, and recall the definition
\[B_{KL}^{(n)}(\eps)=\braces*{b\in \Theta \st K(p_0^{(n)},p_b^{(n)})\leq (n\Delta+1)\eps^2, \Var_{b_0}\brackets[\Big]{\log\frac{p_0^{(n)}}{p_b^{(n)}}}\leq (n\Delta+1)\eps^2}.\]

We have the following abstract contraction result, from which we deduce \Cref{thm:ConcretePriorContraction}.
\begin{theorem}\label{thm:AbstractContractionResult}
Consider data $X^{(n)}=(X_{k\Delta})_{0\leq k \leq n}$ sampled from a solution $X$ to \Cref{eqn:sde} under \Cref{assumptions:b,assumptions:sigma,assumption:Delta,assumption:invariant}. Let the true parameter be $b_0$.
Let $\eps_n\to 0$ be a sequence of positive numbers and let $l_n$ be a sequence of positive integers such that, for some constant $L$ we have, for all $n$,
\begin{equation}\label{eqn:dn-kn-epsn} D_{l_n}=2^{l_n} \leq L n\Delta \eps_n^2, \quad \text{and} \quad n\Delta\eps_n^2/\log(n\Delta) \to\infty. \end{equation}
For each $n$ let $\Theta_n$ be $\Ss$-measurable and assume \begin{equation}\label{eqn:ThetaN} b_0\in \Theta_n\subseteq\braces{b\in \Theta : \norm{\pi_{l_n}b-b}_2 \leq \eps_n},\end{equation} where $\pi_{l_n}$ is the $L^2$--orthogonal projection onto $S_{l_n}$ as described in \Cref{sec:ApproxSpaces}. Let $\Pi^{(n)}$ be a sequence of priors on $\Theta$ satisfying
\begin{enumerate}[(a)]
\item \label{abstractpriorbias} $\Pi^{(n)}(\Theta_n^c)\leq e^{-(\omega+4)n\Delta\eps_n^2}$,
\item\label{abstractpriorsmallball} $\Pi^{(n)}(B_{KL}^{(n)}(\eps_n))\geq e^{-\omega n\Delta \eps_n^2}$, 
\end{enumerate}
for some constant\footnote{In fact we can replace the exponent $\omega+4$ in \Cref{abstractpriorbias} with any $B>\omega+1$. We choose $\omega+4$ because it simplifies the exposition and the exact value is unimportant.} $\omega>0$. 
Then $\Pi^{(n)}\brackets*{\braces{b\in \Theta : \norm{b-b_0}_2 \leq M\eps_n} \mid X^{(n)}}\to 1$ in probability under the law $P_{b_0}$ of $X$, for some constant $M=M(\Ii,L_0,\omega,L)$.
\end{theorem}

The proof, given the existence of tests, follows the standard format of Ghosal--Ghosh--van der Vaart \cite{GGV2000}. A main step in the proof of \Cref{thm:AbstractContractionResult} is to demonstrate an evidence lower bound. 
\begin{lemma}[(Evidence lower bound, ELBO)] \label{lem:ELBO} Recall we defined \[B_{KL}^{(n)}(\eps)=\braces*{b\in \Theta \st K(p_0^{(n)},p_b^{(n)})\leq (n\Delta+1)\eps^2, \Var_{b_0}\brackets[\Big]{\log\frac{p_0^{(n)}}{p_b^{(n)}}}\leq (n\Delta+1)\eps^2},\] where $p_b^{(n)}$ is the joint probability density of $X_0,\dots, X_{n\Delta}$ started from the invariant distribution when $b$ is the true parameter and $p_0^{(n)}$ denotes $p_{b_0}^{(n)}$. Let $n\Delta\eps_n^2\to \infty$ and write $B_{KL}^{(n)}$ for $B_{KL}^{(n)}(\eps_n)$. Define the event \[A_n=\braces[\Big]{\int_\Theta (p_b^{(n)}/p_0^{(n)}) \dPi(b)\geq \Pi(B_{KL}^{(n)})e^{-2n\Delta\eps_n^2}}.\]
	Then as $n\to \infty$, $P_{b_0}\brackets*{A_n^c}\to 0.$
\end{lemma}

\begin{proof}
	Write $\Pi'=\Pi/\Pi(B_{KL}^{(n)})$ for the renormalised restriction of $\Pi$ to $B_{KL}^{(n)}$. Then by Jensen's inequality we have
	\[ \int_\Theta (p_b^{(n)}/p_0^{(n)})(X^{(n)})\dPi(b)\geq \Pi(B_{KL}^{(n)})\exp\brackets*{\int_{B_{KL}^{(n)}} \log(p_b^{(n)}/p_0^{(n)})(X^{(n)}))\dPi'(b)}.\]
	
	Write $Z=\int_{B_{KL}^{(n)}}\log(p_b^{(n)}/p_0^{(n)})\dPi'(b)=-\int_{B_{KL}^{(n)}}\log(p_0^{(n)}/p_b^{(n)})\dPi'(b)$. Applying Fubini's Theorem and using the definition of $B_{KL}^{(n)}$, we see that 
	\[ E_{b_0} Z\geq -\sup_{b\in B_{KL}^{(n)}} E_{b_0} \log(p_0^{(n)}/p_b^{(n)}) \geq -(n\Delta+1)\eps_n^2. \]  
	Further, applying Jensen's inequality and twice applying Fubini's Theorem, we see
	\begin{align*}
	\Var_{b_0} Z &= E_{b_0} \brackets*{\int_{B_{KL}^{(n)}} \log(p_b^{(n)}/p_0^{(n)})\dPi'(b)-E_{b_0}Z}^2\\
	&= E_{b_0} \brackets*{\int_{B_{KL}^{(n)}} \sqbrackets[\Big]{ \log(p_b^{(n)}/p_0^{(n)})-E_{b_0}\log(p_b^{(n)}/p_0^{(n)}) }\dPi'(b)}^2\\
	&\leq E_{b_0} \int_{B_{KL}^{(n)}} \brackets*{\log(p_b^{(n)}/p_0^{(n)})-E_{b_0}\log(p_b^{(n)}/p_0^{(n)})}^2\dPi'(b)\\
	&=\int_{B_{KL}^{(n)}} \Var_{b_0} \brackets*{\log(p_0^{(n)}/p_b^{(n)})}\dPi'(b)
	\leq (n\Delta+1) \eps_n^2,
	\end{align*}
	where to obtain the inequality in the final line we have used the bound on the variance of $\log(p_0^{(n)}/p_b^{(n)})$ for $b\in B_{KL}^{(n)}$. 
	
	Together, these bounds on the mean and variance of $Z$ tell us that \[P_{b_0} \brackets*{\exp(Z)<\exp(-2n\Delta\eps_n^2)}\leq P_{b_0}\brackets*{ \abs{Z-EZ}>(n\Delta-1)\eps_n^2}\leq \frac{(n\Delta+1)\eps_n^{2}}{ (n\Delta-1)^2\eps_n^4}, \] where we have applied Chebyshev's inequality to obtain the final inequality. The rightmost expression tends to zero since $n\Delta\eps_n^2\to \infty$ by assumption, and the result follows. 
\end{proof}

\begin{remark}
	The same is true, but with $P_{b_0}(A_n^c)$ tending to zero at a different rate, if we define $A_n$ instead by $A_n=\braces{\int_\Theta (p_b^{(n)}/p_0^{(n)} \dPi(b) \geq \Pi(B_{KL}^{(n)})e^{-Bn\Delta\eps_n^2}}$ for any $B>1$. That is to say, the exact value 2 in the exponent is not important for the proof.
\end{remark}

\begin{proof}[Proof of \Cref{thm:AbstractContractionResult}]
	We write $\Pi$ for $\Pi^{(n)}$. Since $\Pi(\Theta)=1$ by assumption, it is enough to show 
	$E_{b_0} \Pi\brackets*{\braces{b\in \Theta : \norm{b-b_0}_2 > M\eps_n} \mid X^{(n)}}\to 0$.
	
	Observe, for any measurable sets $S$ and $\Theta_n$, any event $A_n$ and any $\braces{0,1}$--valued function $\psi_n$ we can decompose
	\[ \Pi(S \mid X^{(n)})\leq \II_{A_n^c}+ \psi_n +\Pi(\Theta_n^c \mid X^{(n)})\II_{A_n} + \Pi(S\cap \Theta_n \mid X^{(n)})\II_{A_n} (1-\psi_n). \]
	We apply the above to 
	\begin{equation*} S=S_M^{(n)}=\braces{b\in \Theta \st \norm{b-b_0}_2>M\eps_n}, \quad 
	A_n = \braces[\Big]{\int_\Theta (p_b^{(n)}/p_0^{(n)})(X^{(n)})\dPi(b) \geq e^{-(\omega+2)n\Delta \eps_n^2}},
	\end{equation*}
	with $\Theta_n$ as given in the statement of the \namecref{thm:AbstractContractionResult} and with $\psi_n$ the tests given by \Cref{lem:ExistenceOfTests}, noting that the assumptions for \Cref{thm:AbstractContractionResult} include those needed for \Cref{lem:ExistenceOfTests}. 
	We take the expectation and bound each of the terms separately.
	\paragraph{Bounding $E_{b_0} \II_{A_n^c}$:} 
	We have $P_{b_0}(A_n^c)\to 0$ from \Cref{lem:ELBO}, since by assumption $\Pi(B_{KL}^{(n)}(\eps_n))\geq e^{-\omega n\Delta\eps_n^2}$.
	\paragraph{Bounding $E_{b_0} \psi_n$:} This expectation tends to zero by \Cref{lem:ExistenceOfTests}.
	\paragraph{Bounding $E_{b_0}\sqbrackets{\Pi(\Theta_n^c \mid X^{(n)})\II_{A_n}} $:} We have 
	\begin{align*}\Pi(\Theta_n^c \mid X^{(n)})\II_{A_n} &= \frac{\int_{\Theta_n^c} p_b^{(n)}(X^{(n)})\dPi(b)}{\int_\Theta p_b^{(n)}(X^{(n)})\dPi(b)}\II_{A_n}\\
	&= \frac{\int_{\Theta_n^c} (p_b^{(n)}/p_0^{(n)})(X^{(n)})\dPi(b)}{\int_\Theta (p_b^{(n)}/p_0^{(n)})(X^{(n)})\dPi(b)}\II_{A_n} \\
	&\leq e^{(\omega+2)n\Delta\eps_n^2}\int_{\Theta_n^c} (p_b^{(n)}/p_0^{(n)})(X^{(n)})\dPi(b).
	\end{align*}
	Since $E_{b_0}\sqbrackets{(p_b^{(n)}/p_0^{(n)})(X^{(n)})}=E_b\sqbrackets{1}=1$, taking expectations and applying Fubini's Theorem yields $E_{b_0}\sqbrackets{\Pi(\Theta_n^c \mid X^{(n)})\II_{A_n}} \leq  e^{(\omega+2)n\Delta\eps_n^2} \Pi(\Theta_n^c).$ Since we assumed $\Pi(\Theta_n^c)\leq e^{-(\omega+4)n\Delta\eps_n^2}$, we deduce that \[E_{b_0}[\Pi(\Theta_n^c \mid X^{(n)})\II_{A_n}] \leq \exp\brackets*{(\omega+2)n\Delta\eps_n^2-(\omega+4)n\Delta\eps_n^2} \to 0.\]
	
	\paragraph{Bounding $E_{b_0}\sqbrackets{\Pi(S\cap \Theta_n \mid X^{(n)})\II_{A_n} (1-\psi_n)}$:}
	By a similar argument to the above, observe that
	\[ E_{b_0}\sqbrackets{\Pi(S\cap \Theta_n \mid X^{(n)})\II_{A_n} (1-\psi_n)} \leq e^{(\omega+2)n\Delta\eps_n^2}\int_{b\in \Theta_n : \norm{b-b_0}_2>M\eps_n} E_{b}[1-\psi_n(X^{(n)})] \dPi(b).\] The integrand is bounded by $\sup_{b\in \Theta_n : \norm{b-b_0}_2>M\eps_n} E_b [1-\psi_n(X^{(n)})]\leq e^{-Dn\Delta\eps_n^2}$ by construction of the tests $\psi_n$, where by choosing $M$ large enough we could attain any fixed $D$ in the exponential term. Choosing $M$ corresponding to some $D>\omega+2$ we see \[E_{b_0}\sqbrackets{\Pi(S\cap \Theta_n \mid X^{(n)})\II_{A_n} (1-\psi_n)} \to 0.\qedhere\]
\end{proof}

\begin{proof}[Proof of \Cref{thm:ConcretePriorContraction}]
\begin{thmenum}
\item 
We apply \Cref{thm:AbstractContractionResult}. The key idea which allows us to control the bias and obtain this adaptive result with a sieve prior is \emph{undersmoothing}. Specifically, when we prove the small ball probabilities, we do so by conditioning on the hyperprior choosing a resolution $j_n$ which corresponds to the minimax rate $(n\Delta)^{-s/(1+2s)}$ rather than corresponding to the slower rate $(n\Delta)^{-s/(1+2s)}\log(n\Delta)^{1/2}$ at which we prove contraction. This logarithmic gap gives us the room we need to ensure we can achieve the bias condition \Cref{abstractpriorbias} and the small ball condition \Cref{abstractpriorsmallball} for the \emph{same} constant $\omega$. The argument goes as follows.

Write $\bar{\eps}_n^2=(n\Delta)^{-2s/(1+2s)}$ and let $\eps_n^2= (n\Delta)^{-2s/(1+2s)}\log(n\Delta)$. Choose $j_n$ and $l_n$ natural numbers satisfying (at least for $n$ large enough) 
\[\frac{1}{2}n\Delta\bar{\eps}_n^2\leq D_{j_n}=2^{j_n}\leq n\Delta\bar{\eps}_n^2, \qquad  \frac{1}{2}L n\Delta\eps_n^2 \leq D_{l_n}=2^{l_n}\leq Ln\Delta\eps_n^2,\] where $L$ is a constant to be chosen. Note that \Cref{eqn:dn-kn-epsn} holds by definition. Recall now from our choice of approximation spaces in \Cref{sec:ApproxSpaces} that we have $\norm{\pi_m b_0-b_0}_2\leq K(s)\norm{b_0}_{B_{2,\infty}^s}2^{-ms}$. 
For any fixed $L$ we therefore find that for $n$ large enough, writing $K=K(b_0)=K(s)2^s\norm{b_0}_{B_{2,\infty}^s}$, we have
\begin{align*} \norm{\pi_{l_n}b_0-b_0}_2 \leq K(b_0)(Ln\Delta \eps_n^2)^{-s}
= K(Ln\Delta\bar{\eps}_n^2\log(n\Delta))^{-s}
= KL^{-s}\bar{\eps}_n \log(n\Delta)^{-s}   \leq \eps_n. \end{align*}

Similarly, it can be shown that, with $A=A(\Ii)$ the constant of the small ball result (\Cref{thm:SmallBallProbs}) and for $n$ large enough, we have $\norm{b_0-\pi_{j_n}b_0}_2\leq A\eps_n/2.$

Set $\Theta_n=\braces{b_0}\cup (S_{l_n}\cap \Theta)$ and observe that the above calculations show that the bias condition \Cref{eqn:ThetaN} holds (since also for $b\in \Theta_n,$ if $b\not=b_0$ we have $\norm{\pi_{l_n}b-b}_2=0$). 

Next, for the small ball condition \Cref{abstractpriorsmallball}, recall \Cref{thm:SmallBallProbs} tells us that $\braces{b\in \Theta : \norm{b-b_0}_2\leq\nobreak A\eps_n}\subseteq B_{KL}^{(n)}(\eps_n)$ for all $n$ large enough. Thus it suffices to show, for some $\omega>0$ for which we can also achieve \Cref{abstractpriorbias}, that $\Pi(\braces{b\in \Theta : \norm{b-b_0}_2\leq A\eps_n})\geq e^{-\omega n\Delta\eps_n^2}$. Using that $\norm{b-b_0}_2 \leq \norm{b-\pi_{j_n} b_0}_2+\norm{\pi_{j_n}b_0-b_0}_2\leq \norm{b-\pi_{j_n}b_0}_2+A\eps_n/2$, and using our assumptions on $h$ and $\Pi_m$, we see that 
\begin{align*}
\Pi(\braces{b\in \Theta : \norm{b-b_0}_2\leq A\eps_n})&=\sum_m h(m)\Pi_m(\braces{b\in S_m : \norm{b-b_0}_2\leq A\eps_n}), \\
&\geq h(j_n) \Pi_{j_n} \brackets*{\braces{b\in S_{j_n} : \norm{b-\pi_{j_n} b_0}_2\leq A\eps_n/2}}\\
&\geq h(j_n) (\eps_n A\zeta/2)^{D_{j_n}}\\
&\geq B_1 \exp\brackets*{{-\beta_1 D_{j_n}}+ D_{j_n}\sqbrackets{\log(\eps_n) +\log(A\zeta/2)}}\\
&\geq B_1 \exp\brackets*{-C n\Delta \bar{\eps}_n^2-C n\Delta \bar{\eps}_n^2 \log(\eps_n^{-1})} 
\end{align*}
for some constant $C=C(\Ii,\beta_1,\zeta)$. Since $\log(\eps_n^{-1})=\frac{s}{1+2s}\log(n\Delta)-\frac{1}{2}\log \log (n\Delta)\leq \log(n\Delta),$
we deduce that $\Pi(\braces{b\in \Theta : \norm{b-b_0}_2\leq A\eps_n})\geq B_1e^{-C'n\Delta \bar{\eps}_n^2\log(n\Delta)}=B_1e^{-C'n\Delta \eps_n^2},$ with a different constant $C'$. Changing constant again to some $\omega=\omega(\Ii,\beta_1,B_1,\zeta)$, we absorb the $B_1$ factor into the exponential for large enough $n$.

For \Cref{abstractpriorbias}, since $\Pi(\Theta^c)=0$ by assumption, we have $\Pi(\Theta_n^c)\leq \Pi(S_{l_n}^c)=\sum_{m=l_n+1}^\infty h(m).$ We have assumed that $h(m)\leq B_2e^{-\beta_2 D_m}$, which ensures that the sum is at most a constant times $e^{-\beta_2 D_{l_n}}\leq e^{-\frac{1}{2}L\beta_2n\Delta \eps_n^2}$. For the $\omega=\omega(\Ii,\beta_1,B_1,\zeta)$ for which we proved \Cref{abstractpriorsmallball} above, we can therefore choose $L$ large enough to guarantee $\Pi(\Theta_n^c)\leq e^{-(\omega+4)n\Delta\eps_n^2}$.

\item  Let $\eps_n$ and $j_n$ be as in the statement of the \namecref{thm:ConcretePriorContraction} and define $l_n$ as above (here we can take $L=1$). Similarly to before, we apply results from \Cref{sec:ApproxSpaces} to see 
\[\begin{rcases*} \norm{\pi_{l_n}b-b}_2\leq \eps_n \\ \norm{\pi_{j_n}b-b}_2\leq \eps_n\end{rcases*} \text{ for all $n$ sufficiently large and all $b\in\Theta_s(A_0)$},\]

Set $\Theta_n=\Theta_s(A_0)$ for all $n$. Our assumptions then guarantee the bias condition \Cref{abstractpriorbias} will hold for any $\omega$ (indeed, $\Pi^{(n)}(\Theta_n^c)=0$). Thus it suffices to prove that there exists an $\omega$ such that $\Pi^{(n)}(\braces{b\in \Theta_s(A_0) : \norm{b-b_0}_2\leq 3\eps_n})\geq e^{-\omega n\Delta \eps_n^2},$ since we can absorb the factor of 3 into the constant $M$ by applying \Cref{thm:AbstractContractionResult} to $\xi_n=3\eps_n$.

The prior concentrates on $\Theta_s(A_0)$, so that we have $\Pi^{(n)}(\braces{b : \norm{\pi_{j_n}b-b}_2\leq \eps_n})=1$, and $b_0$ lies in $\Theta_s(A_0)$, so that $\norm{\pi_{j_n}b_0-b_0}_2 \leq \eps_n$. Thus
\[
\Pi^{(n)}(\braces{b\in \Theta_s(A_0) : \norm{b-b_0}_2\leq 3\eps_n}) \geq \Pi^{(n)}(\braces{b\in \Theta_s(A_0) : \norm{\pi_{j_n}b-\pi_{j_n}b_0}_2\leq \eps_n}).
\] 
From here the argument is very similar to the previous part (indeed, it is slightly simpler) so we omit the remaining details. \qedhere
\end{thmenum}
\end{proof}

\subsection{Explicit priors: proofs} 
\label{sec:ExplicitPriorProofs}

\begin{proof}[Proof of \Cref{prop:BasicSievePrior}]

We verify that the conditions of \Cref{thm:SievePriorContraction} are satisfied. 
Condition \Cref{sievepriorhyperparameter} holds by construction. 
The $B_{\infty,1}^s$--norm can be expressed as  
\begin{equation} \label{eqn:Binfty1^s_characterisation} \norm{f}_{B_{\infty,1}^s}=\abs{f_{-1,0}}+\sum_{l=0}^\infty 2^{l(s+1/2)} \max_{0\leq k<2^l}{\abs{f_{lk}}}, \end{equation} (see \cite{Gine2016} Section 4.3) so that any $b$ drawn from our prior lies in $B^1_{\infty,1}$ and satisfies the bound $\norm{b}_{B^1_{\infty,1}}\leq (B+1)(2+\sum_{l\geq 1} l^{-2})$. It follows from standard Besov spaces results (eg.\ \cite{Gine2016} Proposition 4.3.20, adapted to apply to periodic Besov spaces) that $b\in C_\per^1([0,1])$, with a $C_\per^1$--norm bounded in terms of $B$. Thus $\Pi(\Theta)=1$ for an appropriate choice of $K_0$. We similarly see that $b_0\in\Theta$.
It remains to show that \Cref{sievepriorsmallball} holds.
We have \begin{align*} \norm{b-\pi_m b_0}_2^2=\sum_{\substack{-1\leq l< m\\0\leq k <2^l}} \tau_l^2(u_{lk}-\beta_{lk})^2  &\leq \brackets[\Big]{1+\sum_{l=0}^{m-1} 2^{-2l}}\max_{\substack{-1\leq l<m, \\ 0\leq k<2^l}}\abs{u_{lk}-\beta_{lk}} ^2 < 4 \max_{\substack{-1\leq l< m, \\ 0\leq k<2^l}} \abs{u_{lk}-\beta_{lk}}^2, \end{align*} 
so that $\Pi(\braces{b\in S_m : \norm{b-\pi_m b_0}_2\leq \eps})\geq \Pi(\abs{u_{lk}-\beta_{lk}}\leq \eps/2 ~~\forall l,k,-1\leq l< m,k<2^l).$ 
Since we have assumed $\abs{\beta_{lk}}\leq B\tau_l$ and $q(x)\geq \zeta$ for $\abs{x}\leq B$, it follows from independence of the $u_{lk}$ that the right-hand side of this last expression is lower bounded by $(\eps \zeta/2)^{D_m},$ so that \Cref{sievepriorsmallball} holds with $\zeta/2$ in place of $\zeta$. 
\end{proof}

\begin{proof}[Proof of \Cref{prop:KnownSmoothnessPrior}]
	We verify the conditions of \Cref{thm:sK0KnownContraction}. Since $s> 1$ similarly to the proof of \Cref{prop:BasicSievePrior} we see $\Pi^{(n)}(\Theta)=1$ and $b_0\in\Theta$ for an appropriate choice of $K_0$. Observe also that for $A_0=2B+2$ we have $\Pi^{(n)}(\Theta_s(A_0))=1$ by construction, and $b_0\in\Theta_s(A_0)$ by \Cref{assumption:b0inB_infty1^1}, using the wavelet characterisation \Cref{eqn:B2inftywaveletcharacterisation} of $\norm{\cdot}_{B_{2,\infty}^s}$. Thus \Cref{smoothnesspriorsupport} holds and it remains to check \Cref{smoothnesspriorsmallball}.
	
	Let $j_n\in \NN$ be such that $j_n\leq \bar{L}_n$, $2^{j_n}\sim (n\Delta)^{1/(1+2s)}.$ Similarly to the proof of \Cref{prop:BasicSievePrior} we have
	\begin{equation*}\Pi^{(n)}(\braces{b\in \Theta : \norm{\pi_{j_n}b-\pi_{j_n} b_0}_2\leq \eps_n}) \geq \Pi^{(n)}(\abs{u_{lk}-\beta_{lk}}\leq \eps_n/2  ~~\forall l< j_n,~\forall k<2^l) \geq (\eps_n\zeta/2 )^{D_{j_n}},\end{equation*}
	so we're done.
\end{proof}


\begin{proof}[Proof of \Cref{prop:InvariantDensityPrior}]
	We include only the key differences to the previous proofs.
	
	Adapting slightly the proof of \Cref{prop:BasicSievePrior}, we see that $H$ and $H_0$ both have $B_{\infty,1}^2$--norm bounded by $(B+1)(2+\sum_{l\geq 1} l^{-2}).$ Since $\norm{b}_{C^1_\per}\leq \frac{1}{2}\norm{\sigma^2}_{C^1_\per}\brackets{1+\norm{H}_{C^2_\per}}$ and using \cite{Gine2016} Proposition 4.3.20, adapted to apply to periodic Besov spaces, to control $\norm{H}_{C^2_\per}$ by $\norm{H}_{B^2_{\infty,1}}$, we see that for some constant $K_0=K_0(B)$ we have $b_0\in \Theta(K_0)$ and $\Pi^{(n)}(\Theta(K_0))=1$. 
	From the wavelet characterisation 
	\begin{equation*} 
	\norm{f}_{B_{2,2}^s}= \abs{f_{-1,0}}+\brackets[\Big]{ \sum_{l=0}^\infty 2^{2ls} \sum_{k=0}^{2^l-1} f_{lk}^2}^{1/2} \end{equation*} it can be seen that $H$ and $H_0$ have Sobolev norm $\norm{\cdot}_{B_{2,2}^{s+1}}$ bounded by some $A_0'$, hence for some constant $K=K(A_0',s)$ we have $\norm{H-\pi_m H}_{B_{2,2}^1}\leq K2^{-ms}$ and similarly for $H_0$. Since the $B_{2,2}^{s+1}$ norm controls the $B^{s+1}_{2,\infty}$ norm, and we have assumed $\sigma^2\in \Theta_{s+1}$, we additionally see that $b_0\in \Theta_s(A_0)$ and $\Pi^{(n)}(\Theta_s(A_0))=1$ for an appropriate constant $A_0$. Note that here we also depend on the assumption $\sigma^2\in C^s$ to allow us to control $\norm{b}_{B^s_{2,\infty}}$: Remark 1 on page 143 of Triebel \cite{Triebel1983} and Proposition 4.3.20 from \cite{Gine2016} together tell us that $\norm{\sigma^2H'}_{B^s_{2,\infty}}\leq c \norm{\sigma^2}_{C^\alpha} \norm{H'}_{B^s_{2,\infty}}$ for some constant $c=c(s)$, and similarly for $H_0$.

	Observe, for $j_n\in\NN$ such that $j_n\leq \bar{L}_n$ and $2^{j_n}\sim (n\Delta)^{1/(1+2s)}$, \begin{align*}
	\norm{\pi_{j_n}b-\pi_{j_n}b_0}_2 &\leq \norm{b-b_0}_2 =\norm{\sigma^2 (H'-H_0')/2}_2 \leq \frac{1}{2}\sigma_U^2 \norm{H-H_0}_{B_{2,2}^1} \\
	& \qquad \leq \frac{\sigma_U^2}{2}\brackets[\Big]{ \norm{H-\pi_{j_n}H}_{B_{2,2}^1}+\norm{H_0-\pi_{j_n}H_0}_{B_{2,2}^1}+\norm{\pi_{j_n}H-\pi_{j_n}H_0}_{B_{2,2}^1}}. 
	\end{align*}
	Now $\sigma_U^2\norm{H-\pi_{j_n}H}_{B_{2,2}^1}\leq \sigma_U^2 K 2^{-j_n s} \leq C (n\Delta)^{-s/(1+2s)}\leq \frac{1}{2}(n\Delta)^{-s/(1+2s)}\log(n\Delta)^{1/2}=\frac{1}{2} \eps_n$ for large enough $n$, and similarly for $H_0$.
	
	Thus, \begin{align*}\Pi^{(n)}\brackets[\Big]{\braces[\big]{b : \norm{\pi_{j_n}b-\pi_{j_n}b_0}_2\leq \eps_n}} &\geq \Pi^{(n)} \brackets[\Big]{\braces[\big]{b : \norm{\pi_{j_n}H-\pi_{j_n}H_0}_{B_{2,2}^1}\leq \sigma_U^{-2}\eps_n/2}} \\
	&\geq \Pi^{(n)}(\abs{u_{lk}-\beta_{lk}}\leq \kappa \eps_n ~~\forall l<j_n,~\forall k<2^l),
	\end{align*}
	where the final inequality can be seen to hold from the wavelet representation of $\norm{\cdot}_{B_{2,2}^1}$ (the constant $\kappa$ can be taken to be $\kappa=\frac{1}{2}\sigma_U^{-2} (1+(\sum_{k=0}^\infty 2^{-2l})^{1/2})^{-1}>\sigma_U^{-2}/6)$.
	The small ball condition \Cref{smoothnesspriorsmallball} follows from our updated assumptions.
\end{proof}

\section*{Acknowledgements}
This work was supported by the UK Engineering and Physical Sciences Research Council (EPSRC) grant EP/L016516/1 for the University of Cambridge Centre for Doctoral Training, the Cambridge Centre for Analysis. I would like to thank Richard Nickl for his valuable support throughout the process of writing this paper. I would also like to thank two anonymous referees for their very helpful suggestions.

\begin{appendices}\label{sec:appendices}
\crefalias{section}{appsec}

\section{Technical lemmas} \label{sec:TechnicalLemmas}

\begin{lemma}\label{lem:abstractBayes}
Let $\QQ,\PP$ be mutually absolutely continuous probability measures and write $f=\od{\QQ}{\PP}$. Then, for any measurable $g$ and any sub--$\sigma$--algebra $\Gg$, ${E_\QQ \sqbrackets{g \mid \Gg}}=\frac{E_\PP[fg \mid \Gg]}{E_\PP[f \mid \Gg]}. $
\end{lemma}
\begin{proof}
This follows straightforwardly using the characterisation of conditional expectation in terms of expectations against $\Gg$--measurable functions. Precisely, we recall that \begin{equation*}\label{eqn:CondExpCharacterisation} \tag{$\star$}
E_\PP [c(X)v(X) ]=E_\PP [u(X)v(X)] \end{equation*} holds for any $\Gg$--measurable function $v$ if $c(X)=E_\PP[u(X)\mid \Gg]$ a.s., and conversely if $c(X)$ is $\Gg$--measurable and \Cref{eqn:CondExpCharacterisation} holds for any $\Gg$--measurable $v$ then $c(X)$ is a version of the conditional expectation $E_\PP[u(X)]$. For the converse statement it is in fact enough for  \Cref{eqn:CondExpCharacterisation} to hold for all indicator functions $v=\II_A$, $A\in\Gg$.

Applying \Cref{eqn:CondExpCharacterisation} repeatedly we find, for $A\in \Gg$,
\[E_\PP\sqbrackets*{ E_\QQ [g \mid \Gg]E_\PP[f \mid \Gg]\II_A} = E_\PP\sqbrackets*{fE_\QQ[g\mid \Gg] \II_A}=E_\QQ \sqbrackets*{E_\QQ[g\mid\Gg] \II_A} =E_\QQ \sqbrackets*{g\II_A}=E_\PP\sqbrackets*{fg\II_A},
\]
so that, since also $E_\QQ [g \mid \Gg]E_\PP[f \mid \Gg]$ is $\Gg$-measurable, it is (a version of) $E_\PP\sqbrackets*{fg \mid \Gg}$, as required.
\end{proof}

\begin{lemma}\label{lem:VarianceTensorises}
The variance of the log likelihood ratio tensorises in this model, up to a constant. Precisely,
$\Var_{b_0} \log\brackets*{\frac{p_0^{(n)}(X^{(n)})}{p_b^{(n)}(X^{(n)})}}
\leq 3\Var_{b_0}\brackets*{\log\frac{\pi_0(X_0)}{\pi_b(X_0)}} +3n\Var_{b_0}\brackets*{\log\frac{p_0(X_0,X_\Delta)}{p_b(X_0,X_\Delta)}}. $\end{lemma}

\begin{proof}
We write $\log\brackets[\Big]{\frac{p_0^{(n)}(X^{(n)})}{p_b^{(n)}(X^{(n)})}}=U+V+W,$
where $U=\log\frac{\pi_0(X_0)}{\pi_b(X_0)}$ and 
\[ V=\sum_{\substack{1\leq k \leq n \\ k\text{ odd}}} \log\frac{p_0(\Delta,X_{(k-1)\Delta},X_{k\Delta})}{p_b(\Delta,X_{(k-1)\Delta},X_{k\Delta})}, \qquad W=\sum_{\substack{1\leq k \leq n \\ k\text{ even}}} \log\frac{p_0(\Delta,X_{(k-1)\Delta},X_{k\Delta})}{p_b(\Delta,X_{(k-1)\Delta},X_{k\Delta})}.\]
Note now that $V$ and $W$ are both sums are of \emph{independent} terms since $(X_{k\Delta})_{k\leq n}$ is a Markov chain. We thus have
\[\Var_{b_0}(V)=\#\braces{1\leq k \leq n : k\text{ odd}} \Var_{b_0}\brackets*{\log\frac{p_0(X_0,X_\Delta)}{p_b(X_0,X_\Delta)}},\] and a corresponding result for $W$. Using $\Var(R+S+T)=\Var(R)+\Var(S)+\Var(T)+2\Cov(R,S)+2\Cov(S,T)+2\Cov(T,R)$ and $2\Cov(R,S)\leq \Var(R)+\Var(S)$, one derives the elementary inequality \(\Var(U+V+W)\leq 3(\Var(U)+\Var(V)+\Var(W)).\)
The result follows. 
\end{proof}

\begin{lemma}\label{lem:DiffusionBridge}
Let $\tilde{p}_0$ be as in \Cref{eqn:tildep}. Let $p^*(\Delta,x,y)$ be the density of transitions from $x$ to $y$ in time $\Delta$ for a process $U\sim \WW_\sigma^{(x)}$. Then
\[\frac{p_0(\Delta,x,y)}{p_*(\Delta,x,y)}=E_{\WW_\sigma^{(x)}}\sqbrackets*{\tilde{p}_0(U) \mid U_\Delta =y}.\]
\end{lemma}

\begin{proof}
Let $U\sim\WW_\sigma^{(x)}$ and let $\BB_\sigma^{(x,y)}$ denote the law on $C([0,\Delta])$ of $U$ conditional on $U_\Delta=y$. We define the conditional law rigorously via disintegration (eg.\ see \cite{Pollard2001} Chapter 5, Theorem 9, applied to $\lambda=\WW_\sigma^{(x)}$, $\Xx=C([0,\Delta])$ with the sup norm, $T((U_t)_{t\leq \Delta})=U_\Delta$ and $\mu(\dy)=p^*(\Delta,x,y)\dy$), so that 
\begin{equation*}\label{eq:disintegration}
E_{\WW_\sigma^{(x)}}[f(U)]=\int_{-\infty}^\infty p^*(\Delta,x,y) E_{\BB_\sigma^{(x,y)}}[f(U)]\dy, \end{equation*}
for all non-negative measurable functions $f$. 
Taking $f(U)=\tilde{p}_0(U)\II\braces{U_\Delta\in A}$ for an arbitrary Borel set $A\subseteq \RR$, we see
\[ P_{b_0}^{(x)}(X_\Delta\in A) = \int_{-\infty}^{\infty} p^*(\Delta,x,y)\II\braces{y\in A}E_{B_\sigma^{(x,y)}}[\tilde{p}_0]\dy.\] The result follows.
\end{proof}

\section{Proofs for \Cref{sec:HolderProperties}}\label{sec:HolderPropertiesProofs}

\begin{proof}[Proof of \Cref{lem:X_holder}]
	Set $Y_t=S(X_t)$, where 
	\[S(x) = \int_0^x \exp\brackets[\Big]{-\int_0^y \frac{2b}{\sigma^2} (z)\dz}\dy\]
	is the scale function, and let $\psi$ be the inverse of $S$. 
	Since $S''$ exists and is continuous, It{\^o}'s formula applies to yield 
	\[\dY_t=\tilde{\sigma}(Y_t)\dW_t, \quad \tilde{\sigma}(y):=S'(\psi(y))\sigma(\psi(y)).\]
	Let $A=A(\Ii)=\max\brackets{\sigma_U^2 \exp(4K_0/\sigma_L^2),1}$ and observe that $\norm{\tilde{\sigma}^2}_\infty \leq A$. Thus, there are constants $C=C(\Ii)$ and $\lambda=\lambda(\Ii)$ so that for any $u>C\max(\log m,1)^{1/2}$, the event
		\[\Dd=\braces*{ \sup \braces[\bigg]{\frac{\abs{Y_t-Y_s}}{w_m(\abs{t-s})} : {s,t\in [0,m],~s\not= t,~\abs{t-s}\leq A^{-1}e^{-2}}}\leq u},\]
	occurs with probability at least $1-2e^{-\lambda u^2}$, by \Cref{lem:Y_holder}. Now $X_t=\psi(Y_t)$ and $\psi$ is Lipschitz with constant $\norm{\psi'}_\infty=\norm{1/(S'\circ \psi)}_\infty \leq \exp(2K_0\sigma_L^{-2})$. 
	It follows that on $\Dd$, writing $\tau=A^{-1}e^{-2}$, we have for any $s,t\in[0,m]$, $s\not=t$, $\abs{t-s}\leq \tau$,
	\[\abs{X_t-X_s} \leq \exp(2K_0\sigma_L^{-2}) \abs{Y_t-Y_s}\leq \exp(2K_0\sigma_L^{-2})w_m(\abs{t-s})u	\]
	The result follows by relabelling $(\exp(2K_0/\sigma_L^2)u)\mapsto u$, $\lambda\mapsto \lambda\exp(-4K_0/\sigma_L^2)$ and $C\mapsto C\exp(2K_0/\sigma_L^2)$.
\end{proof}

\begin{proof}[Proof of \Cref{lem:Y_holder}]
Recall $w_m(\delta):=\delta^{1/2}(\log(\delta^{-1})^{1/2}+\log(m)^{1/2})$ for $m\geq 1$ and $w_m(\delta):=w_1(\delta)$ for $m<1$. We see that we may assume $m\geq 1$ and the result for $m<1$ will follow. By the \mbox{(Dambis--)}Dubins-Schwarz Theorem (Rogers \& Williams \cite{Rogers2000}, (34.1)),
we can write $Y_t=Y_0+B_{\eta_t}$ for $B$ a standard Brownian motion and for $\eta_t=\qv{Y}_t$ the quadratic variation of $Y$.  
Define the event
\[\Cc=\braces*{ \sup\braces[\bigg]{\frac{\abs{B_{t'}-B_{s'}}}{w_{Am}(\abs{t'-s'})} : ~s',t'\in[0,Am],~s'\not=t',~ \abs{t'-s'}\leq e^{-2}
}\leq u}.\]
By \Cref{lem:B_holder}, there are universal constants $C$ and $\lambda$ so that for $u>C\max(\log(Am),1)^{1/2}$, $\Cc$ occurs with probability at least $1-2e^{-\lambda u^2}$, and note that by allowing $C$ to depend on $A$ we can replace $\max(\log(Am),1)$ with $\max(\log(m),1)$. On this event, for $s,t \in [0,m]$ with $\abs{t-s}\leq A^{-1}e^{-2}$ and $s\not=t$ we have
\begin{align*}
\abs{Y_t-Y_s}&=\abs{B_{\eta_t}-B_{\eta_s}} \\
& \leq \sup\braces{\abs{B_{t'}-B_{s'}} : ~s',t' \in [0,Am],~s'\not= t',~\abs{t'-s'}\leq A \abs{t-s}} \\
& \leq u\sup \braces{ w_{Am}(\abs{t'-s'}) : ~s',t'\in[0,Am],~s'\not= t',~\abs{t'-s'}\leq A\abs{t-s}}\\
& \leq w_{Am}(A \abs{t-s})u ,
\end{align*}
where we have used that $w_{Am}(\delta)$ is increasing in the range $\delta\leq e^{-2}$ 
to attain the final inequality.
Recalling we assume $A\geq 1$, one sees that $w_{Am}(A \delta)\leq A^{1/2}w_{Am}(\delta)$ provided $\delta\leq A^{-1}$, 
which holds in the relevant range. Thus, on $\Cc$, and for $s,$ $t$ and $u$ in the considered ranges,
\begin{align*}
\abs{Y_t-Y_s}&\leq A^{1/2}u \abs{t-s}^{1/2}\brackets*{(\log(Am))^{1/2}+(\log\abs{t-s}^{-1})^{1/2}}\\
&\leq A' u \abs{t-s}^{1/2} \brackets*{(\log(m))^{1/2}+(\log \abs{t-s}^{-1})^{1/2}},
\end{align*} where $A'$ is a constant depending on $A$ (note we have absorbed a term depending on $\log(A)$ into the constant, using that $\log(\abs{t-s}^{-1})\geq 2$). The desired result follows upon relabelling $A'u\mapsto u$ since $C$ and $\lambda$ are here allowed to depend on $A$.

For the particular case $\dY_t=\tilde{\sigma}(Y_t)\dW_t$, we simply observe that $\abs{\qv{Y}_t-\qv{Y}_s}=\abs{\int_s^t \tilde{\sigma}^2(Y_s)\ds}\leq \norm{\tilde{\sigma}^2}_\infty \abs{t-s}$. 
\end{proof}

\begin{proof}[Proof of \Cref{lem:B_holder}]
Assume $m\geq 1$; the result for $m<1$ follows. For a Gaussian process $B$, indexed by $T$ and with intrinsic covariance {(pseudo\nobreakdash-\hspace{0pt})}metric $d(s,t)=(E\sqbrackets{ \brackets{B_t-B_s}^2})^{1/2}$, Dudley's Theorem (\cite{Gine2016} Theorem 2.3.8) says
	\[E\sqbrackets*{ \sup_{s,t\in T, s\not=t} \frac{\abs{B_t-B_s}}{\int_0^{d(s,t)} \sqrt{\log N(T,d,x)}\dx}} <\infty,\] where $N(T,d,x)$ is the number of (closed) balls of $d-$radius $x$ needed to cover $T$. 
	Inspecting the proof, it is in fact shown that the process
	\[ C_u=\frac{B_{u_2}-B_{u_1}}{\int_0^{d(u_1,u_2)}\sqrt{\log(N(T,d,x))}\dx} \quad \text{on} \quad U=\braces{u=(u_1,u_2) : u_1,u_2\in T,~d(u_1,u_2)\not=0},\] is a Gaussian process on  with bounded and continuous sample paths.
	It follows by \cite{Gine2016} Theorem 2.1.20 that 
	\[\Pr\braces*{\abs*{\sup_{v\in V} \abs{C_v}-E\sup_{v\in V} \abs{C_v}}>u}\leq 2e^{-u^2/2\sigma^2},\] for any subset $V$ of $U$, where $\sigma^2=\sup_{v\in V} E [C_v^2].$ 
	We can upper bound $C_v$ by applying the trivial lower bound for the denominator $\int_0^a \sqrt{\log N(T,d,x)} \geq \frac{a}{2}\sqrt{\log 2}$ for any $a=d(u,v)$ with $u,v\in T$ (this follows from the fact that $N(T,d,x)\geq 2$ if $x$ is less than half the diameter of $T$). Using also that $d$ is the intrinsic covariance metric, we deduce that $E C_v^2 \leq 4/\log 2$,  so we can take $\sigma^2=4/\log 2$.

	We will apply the result to $B$ a standard Brownian motion on $T=[0,m]$, which has intrinsic covariance metric $d(s,t)=\abs{t-s}^{1/2}$. For this $T$ and $d$, we have $N(T,d,x)\leq mx^{-2}$. 
	Then, applying Jensen's inequality, we see 
	\begin{align*} \int_0^{d(s,t)} \sqrt{\log N(T,d,x)}\dx &\leq d(s,t)^{1/2} \brackets[\Big]{\int_0^{d(s,t)} \log(N(T,d,x)}^{1/2}\\
	& \leq 2^{1/2}d(s,t) \sqbrackets*{1+\log (d(s,t)^{-1})+\log m}^{1/2}.
	\end{align*}
	Set $V=\braces{u=(s,t)\in U : \abs{t-s}\leq e^{-2}}$ and observe that for $(s,t)\in V$ we have
	$1+\log (d(s,t)^{-1})=1+\frac{1}{2} {\log (\abs{t-s}^{-1})}\leq \log(\abs{t-s}^{-1}).$ 
	Noting further that $(a+b)^{1/2}\leq a^{1/2}+b^{1/2}$ for $a,b\geq 0$ and recalling we defined $w_m(\delta)=\delta^{1/2} \brackets{\brackets{\log \delta^{-1}}^{1/2}+\log(m)^{1/2}},$ we see \[\int_0^{d(s,t)} \sqrt{\log N(T,d,x)}\dx \leq 2^{1/2}w_m(\abs{t-s}).\] Thus, writing $M=E\sqbrackets*{ \sup\braces[\Big]{\frac{\abs{B_t-B_s}}{\int_0^{d(s,t)} \sqrt{\log N(T,d,x)}\dx} : s,t\in T, s\not=t, \abs{t-s} \leq e^{-2}}} $ we see
	\[\Pr\sqbrackets[\Big]{ \sup \braces[\Big]{ \frac{\abs{B_t-B_s}}{w_m(\abs{t-s})} : s,t\in T, s\not=t, \\ \abs{t-s} \leq e^{-2}}>2^{1/2}(M+u)} \leq 2e^{-(u^2(\log 2)/8)}.\] 
	As $M$ is a fixed finite number, we can write $M+u=(1+\eps) u$ with $\eps \to 0$ as $u\to \infty$. Then \[\Pr\sqbrackets*{ {\sup_{\substack{s,t\in T, s\not=t, \\ \abs{t-s} \leq  e^{-2}}}{\frac{\abs{B_t-B_s}}{w_m(\abs{t-s})}}}>u} \leq 2e^{-(u^2 (\log 2)/16(1+\eps)^2)}.\]
	Thus provided $u$ is larger than $M$, we have the result with the constant $\lambda=(\log 2)/64$. 
	
	Finally we track how $M$ grows with $m$ in order to know when $u$ is large enough for this \namecref{lem:B_holder} to apply. Observe that we can write $ M=E \max_k{M_k},$
	where \[ M_k = \sup_{\substack{s,t\in T_k, s\not=t, \\ \abs{t-s}\leq e^{-2}}}  \frac{\abs{B_t-B_s}}{\int_0^{d(s,t)} \sqrt{\log N(T,d,x)}\dx}, \qquad T_k=[ke^{-2},(k+2)e^{-2}].\] As $N(T,d,x)\geq N(T_k,d,x)$, defining 
	\[M_k' = \sup_{s,t\in T_k, s\not=t, \abs{t-s}\leq e^{-2}} \frac{\abs{B_t-B_s}}{\int_0^{d(s,t)} \sqrt{\log N(T_k,d,x)}\dx},\] we see $M_k\leq M_k'$. As with the whole process $C$ we can apply \cite{Gine2016} Theorem 2.1.20 to each $M_k'$ to see that 
	$ \Pr \brackets{\abs{M_k'-E M_k'}>v} \leq 2e^{-v^2/2\sigma^2},$ with $\sigma^2=4/\log 2$ as before. That is, each $(M_k'-E M_k')$ is subgaussian with parameter $12/\sqrt{\log 2}$ (see \cite{Gine2016} Lemma 2.3.1). They all have the same constant (i.e.\ not depending on $m$) expectation, we can bound their maximum, by standard results for subgaussian variables (eg.\ see \cite{Gine2016} Lemma 2.3.4):
	\[
	E M = E \sqbrackets[\big]{EM_0'+ \max_k\braces{M_k'-E M_0'}} \leq E M_0' + 12\sqrt{2\log N/\log 2},
	\]
	where $N$ is the number of $M_k'$ over which we take the maximum and scales linearly with $m$. It follows that $M$ is of order bounded by $\sqrt{\log(m)}$ as $m\to \infty$. 
\end{proof}

\section{Notation}\label{sec:notation}
We collect most of the notation used in the course of this paper. 
\small
\begin{description}
\itemsep -0.1em 
\item $X$: A solution to $\dX_t=b(X_t)\dt+\sigma(X_t)\dW_t$. 
\item $\dot{X}$:   The periodised diffusion $\dot{X}=X \mod 1$. 
\item$b,\sigma$: Drift function, diffusion coefficient.
\item $\mu=\mu_b$; $\pi_b$: Invariant distribution/density of $\dot{X}$.
\item $P_b^{(x)}$: Law of $X$ on $C([0,\infty])$ (on $C([0,\Delta])$ in \Cref{sec:SmallBallProbs}) for initial condition $X_0=x$.
\item $E_{b}$; $P_b$; $\Var_b$:  Expectation/probablity/variance according to the law of $X$ started from $\mu_b$.
\item $E_\mu;\Var_\mu$, and similar: Expectation/variance according to the subscripted measure.
\item $\WW_\sigma^{(x)}$: Notation for $P_b^{(x)}$ when $b=0$.
\item $p_b(t,x,y),\dot{p}_b(t,x,y)$: Transition densities of $X,\dot{X}$ (with respect to Lebesgue measure). 
\item $\tilde{p}_b$: Density (with respect to $\WW_\sigma^{(x)}$) of $P_b^{(x)}$ on $C([0,\Delta])$.
\item $I_b(x)=\int_0^x (2b/\sigma^2(y))\dy$.
\item $X^{(n)}=(X_0,\dots,X_{n\Delta})$; $x^{(n)}=(x_0,\dots,x_{n\Delta})$;  $p_b^{(n)}(x^{(n)})=\pi_b(x_0)\prod_{i=1}^{n}p_b(\Delta,x_{(i-1)\Delta},x_{i\Delta}).$
\item $b_0$: The true parameter generating the data. 
\item $\mu_0,\pi_0,p_0$ etc.: Shorthand for $\mu_{b_0},\pi_{b_0},p_{b_0}$ etc. 
\item $\sigma_L>0;$ $\sigma_U<\infty$: A lower and upper bound for $\sigma$. 
\item $L_0$: A constant such that $n\Delta^2\log(1/\Delta)\leq L_0$ for all $n$.
\item $\Theta=\Theta(K_0)$: The maximal paramater space: $\Theta=\braces{f\in C_\per^1([0,1]) \st ~\norm{f}_{C_\per^1}\leq K_0}$. 
\item $\Theta_s(A_0)=\braces{f\in \Theta : \norm{f}_{B_{2,\infty}^s}\leq A_0}$, for $B_{2,\infty}^s$ a (periodic) Besov space.
\item $\Ii=\braces{K_0,\sigma_L,\sigma_U}$.
\item $S_m$: Wavelet approximation space of resolution $m$, generated by periodised Meyer-type wavelets: $S_m=\Span \braces{\psi_{lk}: -1\leq l< m,0\leq k<2^l},$ where $\psi_{-1,0}$ is used as notation for the constant function~1. 
\item $D_m=\dim(S_m)=2^m$; $\pi_m=$($L^2$--)orthogonal projection onto $S_m$.
\item $w_m(\delta)=\delta^{1/2}(\log(\delta^{-1})^{1/2}+\log(m)^{1/2})$ if $m\geq 1$, $w_m:=w_1$ if $m<1$.
\item $\II_A$: Indicator of the set (or event) $A$.
\item $K(p,q)$: Kullback--Leibler divergence between densities $p,q$: $K(p,q)=E_p\sqbrackets{\log(p/q)}.$
\item $\KL(b_0,b)=E_{b_0}\log(p_0/p_b)$.
\item $B_{KL}^{(n)}(\eps)=\braces*{b\in \Theta \st K(p_0^{(n)},p_b^{(n)})\leq (n\Delta+1)\eps^2, \Var_{b_0}\brackets[\Big]{\log\brackets[\big]{{p_0^{(n)}}/{p_b^{(n)}}}}\leq (n\Delta+1)\eps^2}.$
\item $B_\eps=\braces*{b\in \Theta \st K(\pi_0,\pi_b)\leq \eps^2,~\Var_{b_0}\brackets[\Big]{\log\frac{\pi_0}{\pi_b}}\leq \eps^2,~\KL(b_0,b)\leq \Delta \eps^2,~\Var_{b_0}\brackets[\Big]{\log\frac{p_0}{p_b}}\leq \Delta\eps^2}.$
\item $\Pi$: The prior distribution.
\item $\Pi(\cdot \mid X^{(n)})$:  The posterior distribution given data $X^{(n)}$.
\item $\ip{\cdot,\cdot}$: the $L^2([0,1])$ inner product, $\ip{f,g}=\int_0^1 f(x)g(x)\dx$.
\item $\norm{\cdot}_2$: The $L^2([0,1])$--norm, $\norm{f}_2^2=\int_0^1 f(x)^2 \dx$. 
\item$\norm{\cdot}_\mu$: The $L^2(\mu)$--norm, $\norm{f}_\mu^2=\int_0^1 f(x)^2 \mu(\dx)=\int_0^1 f(x)^2 \pi_b(x)\dx$.
\item $\norm{\cdot}_\infty$: The $L^\infty$-- (supremum) norm,\footnote{All functions we use will be continuous hence we can take the supremum rather than needing the essential supremum.} $\norm{f}_\infty=\sup_{x\in[0,1]} \abs{f(x)}$. 
\item $\norm{}_{C_\per^1}\!:$ The $C_\per ^1$--norm, $\norm{f}_{C_\per^1}=\norm{f}_\infty+\norm{f'}_\infty$.
\item $\norm{\cdot}_n$: The empirical $L^2$--norm $\norm{f}_n^2=\sum_{k=1}^n f(X_{k\Delta})^2$.
\end{description}
\end{appendices}

\bibliography{/home/lkwa2/Dropbox/Kweku's/Maths/CCA/Mendeley-Bibtex/library}

\end{document}